\documentclass[a4paper, 12pt]{article}
\usepackage[left=3cm, right=2cm, top=2.5cm, bottom=2.5cm]{geometry}

\usepackage{amssymb,amsmath,amsthm,amsfonts,mathrsfs,bbm}
\usepackage[pdftex]{graphicx}
\usepackage{subcaption}
\usepackage{hyperref}
\usepackage{multirow}
\usepackage{cleveref}
\usepackage{listings}
\usepackage[T1]{fontenc}
\usepackage[utf8]{inputenc}
\usepackage[english]{babel}

\usepackage{csvsimple}
\usepackage{tikz}
\usepackage{pgfplots}
\usepackage{booktabs}

\usepackage{algpseudocode}

\usepackage{paralist,enumitem,tabularx}
\usepackage{mathtools}

\usepackage{layout}
\usepackage{fancyhdr}
\usepackage{setspace}

\usepackage{float}

\usepackage{pdfpages}

\usepackage{datetime}

\usepackage[backend=bibtex,maxnames=99]{biblatex} 
\newcommand*\diff{\mathop{}\!\mathrm{d}}

\numberwithin{equation}{section}

\theoremstyle{bolddef}
\newtheorem{definition}{Definition}[section]
\newtheorem{algorithm}[definition]{Algorithm}
\newtheorem{assumption}[definition]{Assumption}

\theoremstyle{definition}

\newtheorem{remark}[definition]{Remark}

\theoremstyle{boldplain}
\newtheorem{lemma}[definition]{Lemma}
\newtheorem{theorem}[definition]{Theorem}
\newtheorem{proposition}[definition]{Proposition}

\newcommand{\xk}{x^k}
\newcommand{\wk}{w^k}
\newcommand{\wkm}{w^{k-1}}
\newcommand{\dk}{d^k}
\newcommand{\tauk}{\tau_k}
\newcommand{\xkm}{x^{k-1}}
\newcommand{\xkp}{x^{k+1}}

\newcommand{\xj}{x^j}
\newcommand{\xjp}{x^{j+1}}
\newcommand{\xjm}{x^{j-1}}

\newcommand{\xhkp}{{\hat x}^{k+1}}
\newcommand{\pmin}{p_{\min}}

\newcommand{\taug}{\underline{\tau}}
\newcommand{\tauhk}{{\hat \tau}_k}
\newcommand{\taumin}{\tau_\mathrm{min}}
\newcommand{\taumax}{\tau_\mathrm{max}}

\newcommand{\R}{\mathbb{R}}
\newcommand{\barR}{\overline{\mathbb{R}}}
\newcommand{\N}{\mathbb{N}}
\newcommand{\X}{\mathcal{E}}
\newcommand{\dom}{\mathrm{dom}}
\newcommand{\Rvalue}{\mathcal{R}}
\newcommand{\dist}{\mathrm{dist}}

\newcommand{\proj}{\mathrm{proj}}
\newcommand{\clball}{\overline{B}}

\newcommand{\tr}{\mathrm{tr}}
\newcommand{\grad}{\mathrm{grad}}
\newcommand{\retr}{\mathrm{R}}

\date{\today}

\title{A Nonmonotone Descent Method for Optimization Problems
Defined by Upper-$\mathcal{C}^2 $ Functions over Submanifolds}

\author{Christian Kanzow\footnote{University of W\"urzburg, Institute
	of Mathematics, Emil-Fischer-Str.\ 30, 97074 W\"urzburg, Germany; 
	e-mail: christian.kanzow@uni-wuerzburg.de}
	\and Leo Lehmann\footnote{University of W\"urzburg, Institute
		of Mathematics, Emil-Fischer-Str.\ 30, 97074 W\"urzburg, Germany; 
		e-mail: leo.lehmann@uni-wuerzburg.de} }

\begin{document}

\maketitle

\begin{abstract}
We consider the optimization problem of minimizing a nonsmooth function characterized by a nonsmooth formulation of the descent lemma over a manifold. In the unconstrained case over a Euclidean space, this class of functions is called upper-$\mathcal{C}^2$. Using the recent notion of projectional subdifferentials, we show that their descent property carries over to submanifolds. We propose a nonmonotone subgradient method to solve these problems and prove stationarity of accumulation points of the generated sequence as well as convergence and rate-of-convergence results under the Kurdyka-{\L}ojasiewicz property. We also perform numerical experiments and show how our approach can be applied to a certain type of difference of convex functions as well as clustering problems on manifolds.
\end{abstract}

\section{Introduction}

In this paper, we consider the optimization problem
\begin{equation}\label{eq:prob}
	\min_{x \in \mathcal{M}} \varphi(x),
\end{equation}
where $\mathcal{M} \subseteq \X$ is an embedded Riemannian submanifold of some Euclidean space $\X$. The objective $\varphi$ satisfies a nonsmooth (local) formulation of the descent lemma, which holds in particular if $\varphi$ is a function from $\X$ to $\R$ and upper-$\mathcal{C}^2$. This class of functions has been recently recognized as being suitable for extensions of linesearch methods to the nonsmooth setting \cite{aragon2025}. In particular, all functions satisfying the usual smooth descent lemma are upper-$\mathcal{C}^2$, but nonsmooth examples also arise directly in applications. In particular, problems of type \eqref{eq:prob} arise frequently in clustering tasks on manifolds with applications in natural sciences, text processing as well as image, video and time series analysis \cite{Bendokat2024,Cetingul2009,Dhillon2001,Golzy2016,Gruber2006,Jung2021,Oh2019,Sengupta2017,Wiechers2021}.

In the Euclidean setting, a descent method to find a (local) minimum of a continuously differentiable function $ \varphi $ is defined by the iteration
\begin{equation}\label{eq:opt-iteration}
	\xkp := \xk + \tau_k \dk
\end{equation}
where $\dk$ is a descent direction and $\tau_k$ is a stepsize. For optimization problems on nonlinear spaces, i.e.\ manifolds, given a so-called retraction $\retr$, the iteration above can be generalized to
\begin{equation*}
	\xkp = \retr_{\xk}(\tauk \dk).
\end{equation*}
If $\varphi$ is smooth and $\dk$ is the negative of the Riemannian gradient, the iteration is the so called Riemannian gradient descent method. In our case, besides the choice of $\dk$ as a negative subgradient, our framework will be general enough to allow for $\dk$ to be a Newton or quasi-Newton direction. 

The choice of the stepsize $\tauk$ plays an important role for the convergence of the algorithm. Using a backtracking linesearch, $\tauk$ has to be chosen such that a suitable line search criterion holds like the Armijo rule \cite{Armijo1966}. The Armijo condition, however, enforces a descent in the objective function value in every iteration. In contrast, nonmonotone methods allow for controlled increases in the objective function value. Allowing nonmonotonicity often leads to fewer backtracking iterations and thus larger stepsizes, which is the reason why nonmonotone methods possibly perform better in practice \cite{aragon2025, BirginMartinezRaydan2000, DeMarchi2023}.

There are essentially two prominent nonmonotone stepsize rules: The max-type rule by Grippo et al. \cite{GrippoLamparielloLucidi1986} and the mean-type rule by Zhang and Hager \cite{ZhangHager2004}. The work \cite{aragon2025}, which motivated our study of upper-$\mathcal{C}^2$ optimization problems, considers the max-type rule. By employing the mean-rule in our subgradient method later, the global convergence results can be derived without much of the technical difficulties that arise due to the max-rule. Further, the different flavor of nonmonotonicity allows for stronger theoretical guarantees to be derived, namely convergence and convergence rates under the Kurdyka-Lojasiewicz property. To the best of our knowledge, these results are new even in the Euclidean case.

Linesearch methods are also popular in the context of manifold optimization. The nonmonotone mean-type rule in the smooth setting is studied in \cite{Oviedo2022, Qian2024}. In order to ensure convergence, these methods assume some Lipschitz-type conditions \cite{Boumal2018}, which in the Euclidean setting are equivalent to Lipschitz continuity of the gradient or the Lipschitz smoothness property. We propose a more general nonsmooth condition. Similar to the Lipschitz-type conditions, using so-called projectional subdifferentials recently considered in \cite{Khanh2026}, we show that the descent property characterizing the class of upper-$\mathcal{C}^2$ functions can be transferred from the ambient Euclidean space $\X$ to submanifolds.

The paper is organized in the following way. In Section~\ref{Sec:Background}, we introduce some concepts from variational analysis and optimization on manifolds. We further introduce the class of upper-$\mathcal{C}^2$ functions and show that these functions restricted to a submanifold satisfy a certain descent property. In Section~\ref{Sec:Algorithm}, we present our nonmonotone descent method to solve \eqref{eq:prob} and discuss its (global) convergence properties. Subsequently, in Section~\ref{Sec:KL}, we obtain further convergence guarantees under the Kurdyka-Lojasiewicz condition. Next, in Section~\ref{Sec:Numerics}, we are concerned with some applications and conduct corresponding numerical experiments. We conclude with some final remarks in Section~\ref{Sec:Final}.

Notation: We will consider optimization problems on some submanifold $\mathcal{M} \subseteq \X$, where $\X$ is a Euclidean space (finite-dimensional Hilbert space). Throughout this manuscript, we identify the dual space $\X^*$ with $ \X $ itself. We write $ \langle x, y \rangle $ for the scalar product of two elements $ x, y \in \X $ and $ \| x \| $ denotes the induced norm of $ x \in \X $. The distance of a point $ x \in \X $ to a nonempty set $ S \subseteq \X $ is denoted by $ \dist(x, S) := \inf_{y \in  S} \| x - y \| $. Further, for $x \in \X$ and a nonempty closed set $S \subseteq \X$, the projection of $x$ onto $S$ is denoted by $\proj_S(x) := \{ y \in \X \mid \dist(x, S) = \|x-y\| \}$. If $S$ is nonempty, closed and convex, $\proj_S(x)$ is single-valued for all $x \in X$. Also, we denote by $\delta_S$ the indicator function of $S$, which is $0$ for arguments in $S$ and $\infty$ else.

We write $B_r(x)$ for the open and $\clball_r(x)$ for the closed ball of radius $r>0$ around $x \in \X$. For arbitrary $S \subseteq \X$, we denote its interior by $\mathrm{int}(S)$, the closure by $\mathrm{cl}(S)$ and the convex hull by $\mathrm{conv}(S)$.

Finally, $ \R $ denotes the set of real numbers, while $ \barR := (- \infty, 
+ \infty ] $ is the set of extended reals except that we exclude the
value $ - \infty $. Given an extended-valued function $ \theta: 
\X \to \barR $, we call $ \dom (\theta) := 
\{ x \in \X \mid \theta (x)  < \infty \} $
the domain of $ \theta $. The function $ \theta $ is said to
be proper if $ \dom (\theta) $ is nonempty.

\section{Background Material}\label{Sec:Background}

\subsection{Basics from Variational Analysis}

We first recall that a sequence $ \{ \xk \} \subseteq \X $
converges (locally) \emph{Q-linearly} to some limit 
$ x^* \in \X $ if there exists a constant $ c \in (0,1) $ 
such that  $\| \xkp - x^* \| \leq c \| \xk - x^* \|$
holds for all $ k \in \N $ sufficiently large. Furthermore,
we say that $ \{ \xk \} $ converges \emph{R-linearly} to $ x^* $ if
$\limsup_{k \to \infty} \| \xk - x^* \|^{1/k} < 1 $
holds. Note that this property holds if there exist constants
$ \omega > 0 $ and $ \mu \in (0,1) $ such that $ \| \xk - x^* \| \leq 
\omega \mu^k $ for all $ k \in \N $ sufficiently large, i.e.,
if the sequence $ \| \xk - x^* \| $ is dominated by a Q-linearly
convergent null sequence.

We next recall some results from variational analysis and refer the
interested reader to the two monographs \cite{Mordukhovich2018,RockafellarWets2009}
for more details.

Given a proper, lower semicontinuous function $ g: \X \to
\barR $ and any $ x \in \dom (g) $, we call 
\begin{equation*}
	\hat\partial g(x) := \Big\{ v \in \X \, \Big|
	\liminf_{y \to x, y \neq x} \frac{g(y) - g(x) - 
	\langle v, y-x \rangle}{\| y- x \|} \geq 0 \Big\}
\end{equation*}
the \emph{regular} or \emph{Fr\'echet subdifferential} of $ f $ at $ x $, 
whereas
\begin{equation*}
	\partial g(x) := \big\{ v \in \X \mid 
	\exists \xk,  v^k \in \X : \xk \to x, g(\xk) \to g(x),
	v^k \in \hat\partial g(\xk) \ \forall k \big\}
\end{equation*}
is called the \emph{limiting, Mordukhovich}, or 
\emph{basic subdifferential}
of $ g $ at $ x $. 

For a locally Lipschitz continuous function $g$ around $\bar x \in \mathbb{R}$, we define the \emph{Clarke subdifferential} of $g$ at $\bar x $ as 
\begin{equation*}
	\partial_C g(\bar x):= \{ v \in \X \mid \langle v, h \rangle \leq g^\circ (\bar x; h) \text{ for all } h \in \mathbb{R}^n\},
\end{equation*}
where
\begin{equation*}
	g^\circ(\bar x; h) := \limsup_{x \to \bar x, t \to 0^+} \frac{g(x+th) - g(x)}{t}
\end{equation*}
is the \emph{Clarke directional derivative} of $g$ at $\bar x$ in the direction $h$. Note that $\partial_C g(\bar x) = \mathrm{conv}(\partial g(\bar x))$ holds for all locally Lipschitz continuous functions. Further, if $g$ is locally Lipschitz continuous, then the Clarke subdifferentials of $g$ on bounded subsets of $\X$ remain bounded, too, see e.g.\ \cite{clarke1998}. By consequence, as $\partial g(\bar x) \subseteq \partial_C g(\bar x)$ holds, the same is true for the limiting subdifferential.

Let us also recall some notation from variational geometry from \cite{RockafellarWets2009}. For a set $C \subset \X$, we define the \emph{tangent cone} of $C$ at $\bar x \in C$ as 
\begin{equation*}
	T_C(\bar x) := \Big\{ d \in \X\, \Big|\, \exists \{ \xk \} \subset C, t_k \searrow 0, \xk \to \bar x, \frac{\xk - \bar x}{t_k} \to d\Big\}.
\end{equation*}
Further, the \emph{regular normal cone} is given by 
\begin{equation*}
	\hat N_C(\bar x) = T_C(\bar x)^\circ := \{ n \in X \mid \langle n, d \rangle \leq 0 \forall d \in T_C(\bar x)\},
\end{equation*}
and the \emph{limiting normal cone} is defined as 
\begin{equation*}
	N_C(\bar x) = \{ n \in \X \mid \exists \{\xk\} \subset C, \{ n^k\}: \xk \to \bar x, n^k \in \hat N_C(\bar x), n^k \to n\}.
\end{equation*}

Before we proceed by stating some basic concepts and results about manifolds, we finally introduce the so-called Kurdyka–Łojasiewicz property. The following definition is a generalization of the classical 
one for nonsmooth functions, as introduced in \cite{AttouchBolteRedontSoubeyran2010,BolteDaniilidisLewis2007,BolteDaniilidisLewisShiota2007}.
The KL property plays a central role in the local
convergence analysis and rate-of-convergence results of several algorithms for the solution of 
nonsmooth minimization problems, see e.g.\ \cite{attouch2009convergence,Qian2024}.

\begin{definition}\label{def:klproperty}
Let $g: \X \to \overline{\mathbb{R}}$ be lower semicontinuous. We say that $g$ satisfies the {\normalfont Kurdyka–Łojasiewicz (KL) property} at $x^* \in \text{dom}(\partial g) :=
\{x \in \X \, | \, \partial g(x) \neq \emptyset\}$ if there exists a constant $\eta > 0$, a neighborhood $U \subset \X$ of $x^*$, and a continuous and concave function $\chi : [0, \eta] \to [0, \infty)$, called {\normalfont desingularization function}, which is continuously differentiable on $(0, \eta)$ and satisfies $\chi(0)=0$ and $\chi'(t) > 0$ for all $t \in(0, \eta)$, such that the so-called KL inequality 
\begin{equation}\label{eq:klineq}
	\chi' \big( g(x) - g(x^*) \big) \dist \big( 0, \partial g(x)
	\big) \geq 1
\end{equation}
holds for all $x \in U \cap \{x \in \X \, | 
\, g(x^*) < g(x) < g(x^*) + \eta\}$. Furthermore, we call $ g $ a
\emph{KL function} if $ g $ satisfies the KL property at any point
$ x^*  \in \text{dom}(\partial g) $.
\end{definition}

\subsection{Essentials about Optimization on Submanifolds}

As noted in \cite{Boumal2018}, optimization on manifolds is not fundamentally different from optimization in a Euclidean space and many popular methods for unconstrained optimization have been generalized to optimization on manifolds. However, to do so, we require some preliminaries. For a broader overview, we refer to \cite{Absil2008}, \cite{Hu2020} and \cite{Boumal2023}.

Let us recall the following definition of a smooth (embedded) submanifold (see e.g.\ \cite[Example 6.8]{RockafellarWets2009}, \cite[Definition 3.10]{Boumal2023}):
\begin{definition}\label{def:manifold}
	A set $\mathcal{M} \subseteq \mathcal{E}$ is called a \emph{smooth embedded manifold} of dimension $d$ around $\bar x \in \mathcal{M}$ if there exists an open neighborhood $U \subseteq \mathcal{E}$ of $\bar x$ such that $\mathcal{M}$ can be represented as the set of solutions to $F(x) = 0$, where $F: U \to \R^m$ is $\mathcal{C}^1$ with $\nabla F(\bar x)$ of rank $m$, where $m = n-d$. The function $F$ is called a \emph{local defining function} for $\mathcal{M}$ at $\bar x$.
\end{definition}
Note that smooth embedded manifolds are manifolds by themselves, see \cite{Boumal2023}. As we will only be dealing with such manifolds, if we say that a set $\mathcal{M}$ is a manifold, we mean that it satisfies the definition above. Most of the main ideas presented here hold in a more general context and we refer the reader to the references mentioned above for more details. As we will see later, the reason is that the concept of projectional subdifferentials require the more special situation as introduced in Definition~\ref{def:manifold}.

Let $\mathcal{M}$ be a smooth embedded submanifold and $x \in \mathcal{M}$. We follow \cite{Boumal2023} with the following definition and call 
\begin{equation*}
	T_x \mathcal{M} = \{\dot \gamma(0) \mid \gamma : I \to \mathcal{M} \text{ is smooth and } \gamma(0) = x\},
\end{equation*}
the \emph{tangent space} of $\mathcal{M}$ at $x$, where $I$ is an open interval with $0 \in I$. Elements $\xi_x \in T_x\mathcal{M}$ are called \emph{tangent vectors} to $\mathcal{M}$ at $x$. Note that if $F$ is a local defining function of $\mathcal{M}$ at $x$, then by \cite[Theorem 3.15]{Boumal2023}, it holds that $T_x\mathcal{M} = \ker DF(x)$. On the other hand, by \cite[Example 6.8]{RockafellarWets2009}, it also holds that $T_\mathcal{M}(x) = \ker DF(x)$. Hence, the tangent space and the tangent cone agree and the notations $T_x \mathcal{M}$ and $T_\mathcal{M}(x)$ can be used interchangeably. Consequently, the normal cone to a submanifold as the orthogonal complement of the tangent space is a linear space itself, also called \emph{normal space}, that is, we have
\begin{equation*}
	N_\mathcal{M}(x) = (T_x \mathcal{M})^\perp.
\end{equation*}
Note that $T_x \mathcal{M}$ is a vector space with the same dimension $d$ as the manifold $\mathcal{M}$ and with its zero element denoted as $0_x$. 

The collection $T\mathcal{M}:= \cup_{x \in \mathcal{M}} \{x\} \times T_x\mathcal{M}$ denotes the \emph{tangent bundle} of the manifold $\mathcal{M}$.

If the tangent space $T_x \mathcal{M}$ is equipped with a smoothly varying inner product $\langle \cdot, \cdot \rangle_x$ called the \emph{Riemannian metric}, the manifold $\mathcal{M}$ called a \emph{Riemannian manifold}. Let us note that if $\mathcal{M} \subseteq \mathcal{E}$ is a submanifold embedded in some Euclidean space $\mathcal{E}$, then $T_x \mathcal{M}$ can be identified with a linear subspace of $\mathcal{E}$ and the inner product on $\mathcal{E}$ carries over to the tangent spaces in a natural way by means of restriction to the respective subspace. In that case, which we will assume in the following, $\mathcal{M}$ becomes a \emph{Riemannian submanifold} of the Euclidean space $\X$.

For a smooth mapping $G: \mathcal{M} \to \mathcal{N}$ between two embedded submanifolds and $x \in \mathcal{M}$, we call the linear map $DG(x): T_x\mathcal{M} \to T_{G(x)} \mathcal{N}$ defined by
\begin{equation*}
	DG(x)[\xi_x] = \frac{\mathrm{d}}{\mathrm{d} t} G(\gamma(t))\large|_{t = 0} = (G \circ \gamma)'(0),
\end{equation*}
where $\gamma$ is a smooth curve on $\mathcal{M}$ through $x$ at $t=0$ with velocity $\dot \gamma(0) = \xi_x$, the \emph{differential} of $G$ at $x$.

On a manifold $\mathcal{M}$, the \emph{Riemannian gradient} of a smooth map $g: \mathcal{M} \to \R$ is the tangent vector $\grad g(x) \in T_x \mathcal{M}$ such that for all $\xi_x \in T_x \mathcal{M}$ it holds that
\begin{equation*}
	Dg(x)[\xi_x] = \langle \grad g(x), \xi_x\rangle_x.
\end{equation*}
Analogous to the Euclidean setting, the negative of the Riemannian gradient is the direction of steepest descent.

In manifold optimization, a key concept are so-called \emph{retractions}. As noted in \cite{Absil2008}, these have two important purposes: They turn elements of $T_x \mathcal{M}$ into points of $\mathcal{M}$ and secondly, by means of the pullback through the retraction $\retr$, functions on $\mathcal{M}$ can be transformed to functions defined on the vector space $T_x \mathcal{M}$.
\begin{definition}
	Let $\mathcal{M}$ be a manifold. A smooth (i.e.\ at least $\mathcal{C}^2$) mapping $\retr: T \mathcal{M} \to \mathcal{M}$ is called a \emph{retraction} if for all $x \in \mathcal{M}$ the restriction $\retr_x := \retr(x, \cdot) : T_x\mathcal{M} \to \mathcal{M}$ satisfies the following properties:
	\begin{enumerate}
		\item $\retr_x(0_x) = x$, where $0_x$ is the zero element of $T_x\mathcal{M}$.
		\item With the identification $T_0 T_x \mathcal{M} \simeq T_x\mathcal{M}$, it holds that 
			\begin{equation*}
				D\retr_x(0) = \mathrm{id}_{T_x \mathcal{M}},
			\end{equation*}
		where $\mathrm{id}_{T_x \mathcal{M}}$ is the identity mapping on $T_x \mathcal{M}$.
	\end{enumerate}
\end{definition}
One possibility for $\retr$ is the so-called \emph{Riemannian exponential map}, which is, however, not always easy to compute. Finding a \emph{good} retraction is crucial for the design of numerical optimization methods on manifolds. For example, in the Euclidean setting, the canonical retraction is given by $\xi \mapsto x+\xi$ (note that this is well-defined as $T_x \mathcal{E} \simeq \mathcal{E}$) and for the sphere $\mathbb{S}^{n-1}$, we obtain a retraction by means the orthogonal projection mapping
\begin{equation*}
	\retr_x (\xi) := \frac{x+\xi}{\|x+\xi\|}.
\end{equation*}

\begin{remark}\label{remark:boundedefradr}
For simplicity, our subsequent theory will assume that we have a globally defined retraction on the manifold appearing in the constraints of our optimization problem, i.e.\ the retraction $\retr$ is indeed a smooth mapping on $T \mathcal{M}$. However, if the stepsizes of our method are adjusted accordingly, our results still hold true under the weaker condition that the radius of definition of $\retr_x$ remains uniformly bounded away from zero on compact subsets of $\mathcal{M}$.
\end{remark}

Recently, there has been a growing interest in extending the classical concepts from nonsmooth and variational analysis to the context of manifolds. Notably, \cite{Khanh2026} consider so-called projectional subdifferentials on submanifolds. In particular, \cite[Definition 3.1]{Khanh2026} and \cite[Theorem 3.1]{Khanh2026} justify the following definition:
\begin{definition}
	Let $\mathcal{M} \subseteq \X$ be a smooth (embedded) manifold with $\bar x \in \mathcal{M}$ and $g : \X \to \barR$. Then we call
	\begin{equation*}
		\hat \partial_\mathcal{M} g(\bar x) := \proj_{T_{\bar x} \mathcal{M}} \hat \partial (g + \delta_\mathcal{M}) (\bar x)
	\end{equation*}
	the \emph{Fréchet projectional subdifferential} of $g$ at $\bar x$ relative to $\mathcal{M}$, and 
	\begin{equation*}
		\partial_\mathcal{M} g(\bar x) := \proj_{T_{\bar x}\mathcal{M}} \partial (g + \delta_\mathcal{M})(\bar x)
	\end{equation*}
	the \emph{projectional (limiting/basic/Mordukhovich) subdifferential} of $g$ at $\bar x$ relative to $\mathcal{M}$.
\end{definition}
If $g$ is differentiable and $\mathcal{M} \subseteq \mathcal{E}$ an embedded manifold, it holds by \cite[Equation 3.37]{Absil2008} that $\grad g(x) = \proj_{T_x\mathcal{M}} \nabla g(x)$. Hence by \cite[Proposition 3.2]{Khanh2026}, we have $\partial_\mathcal{M} g(x) = \hat \partial_\mathcal{M} g(x) = \{\grad g(x)\}$.

Let us also state the following characterizations of the projectional subdifferentials from \cite{Khanh2026}:
\begin{equation}
	\hat \partial_\mathcal{M} g(\bar x) = \hat \partial (g  + \delta_\mathcal{M})(\bar x) \cap T_{\bar x}\mathcal{M} \text{\quad and \quad} \partial_\mathcal{M} g(\bar x) = \partial (g  + \delta_\mathcal{M})(\bar x) \cap T_{\bar x}\mathcal{M}.\label{eq:projsubdiffchar}
\end{equation}
Any point $x^* \in \dom g \cap \mathcal{M}$ satisfying $0 \in \partial_\mathcal{M} g (x^*)$ is called a \emph{stationary point} of $g$ with respect to $\mathcal{M}$. This can be motivated by the fact that given a local minimizer $x^*$ of $g + \delta_\mathcal{M}$, it holds that $0 \in \hat \partial(g + \delta_\mathcal{M})(x^*)$ and hence $0 \in \hat \partial (g+\delta_\mathcal{M})(x^*) \cap T_{\bar x}\mathcal{M} = \hat \partial_\mathcal{M} g(x^*) \subseteq \partial_\mathcal{M} g(x^*)$.

Let us also note that the projectional limiting subdifferential is \emph{robust} in the following sense (see \cite[Corollary 3.1]{Khanh2026}): Assume that $\{\xk\}$ is a sequence converging to some limit $\bar x$ such that $(g+\delta_\mathcal{M})(\xk) \to (g+\delta_\mathcal{M})(\bar x)$ and $\wk \in \partial_\mathcal{M} g(\xk)$ for all $k \in \N$ converges to some $\bar w$. Then it holds that $\bar w \in \partial_\mathcal{M} g(\bar x)$. This property will play an important role later on in our convergence theory. 

The projectional subdifferential also has the following properties. 
\begin{lemma}\label{lem:projsubdiff}
	Assume that $g$ is locally Lipschitz continuous on an open superset of an embedded manifold $\mathcal{M} \subseteq \X$. Then the following holds for all $\bar x \in \mathcal{M}$.
	\begin{enumerate}
		\item For all $w_\mathcal{M} \in \partial_\mathcal{M} g(\bar x)$, there exists an element $w \in \partial g(\bar x)$ such that $w_\mathcal{M} = \proj_{T_{\bar x} \mathcal{M}} (w)$. In particular, $\partial_\mathcal{M} g$ is locally bounded.
		\item The projectional subdifferential is non-empty, i.e.\ we have $\partial_\mathcal{M} g(\bar x) \neq \emptyset$.
	\end{enumerate}
\end{lemma}
\begin{proof}
As $g$ is locally Lipschitz continuous, the sum-rule for the limiting subdifferential as in \cite[Corallary 2.20]{Mordukhovich2018} gives $\partial(g + \delta_\mathcal{M})(\bar x) \subseteq \partial g(\bar x) + N_\mathcal{M}(\bar x)$. Hence, for every $w_\mathcal{M} \in \partial_\mathcal{M} g(\bar x)$, there exists $w \in \partial g(\bar x)$ and $n \in N_\mathcal{M}(\bar x)$ with $w_\mathcal{M} = \proj_{T_{\bar x} \mathcal{M}} (w+n) = \proj_{T_{\bar x} \mathcal{M}}(w)$, as $N_\mathcal{M}(\bar x) = (T_{\bar x} \mathcal{M})^\perp$. The local boundedness of $\partial_\mathcal{M} g$ now follows directly from the corresponding property of the classical subdifferential for locally Lipschitz functions.

As in \cite{Meng2021}, the non-emptiness of $\partial_\mathcal{M} g(\bar x)$ follows from \cite[Corollary 8.10]{RockafellarWets2009}, which states that as $g+\delta_\mathcal{M}$ is finite and locally lower semicontinuous at $\bar x$, there exists a sequence $\{\xk\}$ with $\xk \to \bar x$, $(g + \delta_\mathcal{M})(\xk) \to (g + \delta_\mathcal{M})(\bar x)$ and $\partial (g+\delta_M)(\xk) \neq \emptyset$. By consequence, we have $\partial_\mathcal{M}g (\xk) \neq \emptyset$ for all $k$ large enough. Further, choosing $v^k \in \partial_\mathcal{M} g(\xk)$ gives a bounded sequence by the first part. Hence, on a subsequence, $v^k$ converges to some $\bar v$ and by the robustness property we have $\bar v \in \partial_\mathcal{M}g(\bar x)$, making $\partial_\mathcal{M}g(\bar x)$ non-empty.
\end{proof}

\begin{remark}
Our algorithmic theory later is formulated in terms of the projectional limiting subdifferential. However, the conclusions remain valid for other subdifferentials. The key properties are the robustness property mentioned above and boundedness of the subdifferential for locally Lipschitz functions. In the Euclidean setting, these are shared by the classical limiting and the Clarke subdifferential, while the Fréchet subdifferential is not robust in this sense. Note that in some applications, Clarke subgradients can be computed (see \cite{aragon2025}), however, the notion of stationarity in terms of the Clarke subdifferential is weaker. For our manifold setting, \cite{Hosseini2018} also introduced a Clarke-type subdifferential on Riemannian manifolds for which also a Riemannian Kurdyka–Łojasiewicz property was developed in \cite{Huang2022}. These subdifferentials can be defined for arbitrary manifolds, not only embedded Riemannian submanifolds. The advantage of projectional subdifferentials is that for submanifolds, properties on the ambient Euclidean space can be transferred to properties of the projectional subdifferential as in Section~\ref{subsec:upperc2subm}.
\end{remark}

\subsection{Upper-$\mathcal{C}^2$ functions}

We next introduce the class of upper-$ \mathcal{C}^2 $ functions that will play
a central role for the design and the convergence analysis of our
descent method. For more details on this class of functions, see
\cite{RockafellarWets2009}.

\begin{definition}\label{def:upperC2}
Let $ U \subseteq \X$ be an open set. We say that a function
$ \varphi: U \to \mathbb{R} $ is {\normalfont upper-$ \mathcal{C}^2 $} on $ U $, if, 
on some neighborhood $ V $ of each $ \bar{x} \in U $, there is a 
representation 
\begin{equation*}
	\varphi (x) = \min_{c \in C} \varphi_c (x),
\end{equation*}
where the functions $ \varphi_c $ are of class $\mathcal{C}^2 $ on $ V $, and 
$ C $ is a compact set (in some topological space) such that 
$ \varphi_c $ and its first- and second-order partial derivatives
depend continuously on $ (x,c) \in V \times C $.
\end{definition}

Taking the discrete topology, it follows, for example, that functions of
the form $ \varphi (x) := \min \big\{ f_1(x), \ldots, f_l(x) \big\} $ with
twice continuously differentiable functions $ f_i: U \to \mathbb{R} $
on some open set $ U \subseteq \X$ are upper-$\mathcal{C}^2 $ functions.
Further examples can be derived from the subsequent characterization
of upper-$\mathcal{C}^2$ functions from the recent 
report \cite[Prop.\ 3.2]{aragon2025} that will be particularly relevant 
for our setting.

\begin{proposition}\label{Prop:CharUpperC2}
Let $ U \subseteq \X $ be an open set and $ \varphi: \X
\to \mathbb{R} $ be locally Lipschitz on $ U $. Then the following statements
are equivalent:
\begin{itemize}
	\item[(a)] $ \varphi $ is upper-$\mathcal{C}^2 $ on $ U $.
	\item[(b)] For each $ \bar{x} \in U $, there exist a constant
	   $ \kappa \geq 0 $ and some neighborhood $ V $ of $ \bar{x} $ such 
	   that
	   \begin{equation}\label{eq:descentproperty}
	   	  \varphi (y) \leq \varphi (x) + \langle w, y-x \rangle + 
	   	  \kappa \| y - x \|^2
	   \end{equation}
	   for all $ x, y \in V $ and all $ w \in \partial \varphi (x) $.
 	\item[(c)] For each $ \bar{x} \in U $, there exists some neighborhood
 	   $ V $ of $ \bar{x} $ such that $ \varphi $ can be expressed as
 	   $ \varphi = g - h $, where $ g $ is differentiable with Lipschitz
		gradient, and $ h $ is Lipschitz and prox-regular (see \cite[Definition 3.27]{RockafellarWets2009}). Indeed, one can 
 	   take $ g = \kappa \| \cdot \|^2 $, for some $ \kappa \geq 0 $, and 
 	   $ h $ to be convex.
\end{itemize}
\end{proposition}

The characterization from part (b) is particularly interesting in our
case. The inequality \eqref{eq:descentproperty} is the counterpart of
the usual descent lemma for smooth functions with Lipschitz gradient,
the constant $ \kappa $ in \eqref{eq:descentproperty} plays the role
of the corresponding Lipschitz constant. We stress, however, that 
$ \kappa $ in \eqref{eq:descentproperty} is a local constant depending
on the given point $ \bar{x} $. Furthermore, it is worth noting that
the inequality \eqref{eq:descentproperty} holds for an arbitrary
subgradient $ w \in \partial \varphi (x) $, not just for a particular 
element from $ \partial \varphi (x) $. This observation is highly important for
our descent method to be well-defined. Finally, we note that 
part (c) shows that upper-$\mathcal{C}^2 $ functions are closely related to the
class of DC-functions (DC = difference-of-convex). In particular, this
characterization provides several further examples of nonsmooth 
upper-$\mathcal{C}^2 $ functions. In this respect, we refer the interested reader
also to the corresponding discussion in \cite{aragon2025}.

\subsection{Descent Property of Upper-$\mathcal{C}^2$ Functions on Submanifolds}\label{subsec:upperc2subm}

Before presenting our algorithm in the next section, we establish the following result, which shows that upper-$\mathcal{C}^2$ functions defined on an open superset of an embedded Riemannian submanifold $\mathcal{M} \subseteq \X$ satisfy a certain descent property. This condition is inspired by the retraction smoothness condition in \cite{Boumal2018}, where it was used for differentiable functions as a generalization to the well-known descent lemma to optimization problems on manifolds. Note again, however, that our Assumption~\ref{descentas} is both nonsmooth and merely local.

\begin{assumption}\label{descentas}
	Let $\mathcal{M}$ be a manifold with retraction $\retr$ and assume that for all $\bar x \in \mathcal{M}$ there exists a constant $\tilde \kappa > 0$, $r>0$ and an open neighborhood $V\subseteq \mathcal{M}$ of $\bar x$, such that the objective function $\varphi: \mathcal{M} \to \R$ from \eqref{eq:prob} satisfies
	\begin{equation}\label{eq:descentpropm}
		\varphi \big(\retr_x(\xi_x)\big) - \big(\varphi(x) + \langle w_\mathcal{M}, \xi_x \rangle\big) \leq \tilde \kappa \|\xi_x\|^2,
	\end{equation}
	for all $x \in V$, $\xi_x \in \clball_r(0_x) \subset T_x \mathcal{M}$ and all $w_\mathcal{M} \in \partial_\mathcal{M} \varphi(x)$. Further, assume that $\partial_\mathcal{M} \varphi$ is locally bounded.
\end{assumption}
Let us show that it is sufficient for $\varphi$ to be upper-$\mathcal{C}^2$ on an open superset of $\mathcal{M}$ for $\varphi$ to satisfy \eqref{eq:descentpropm}. To do so, we first need the following properties of retractions. These are somewhat standard and similar arguments appear in the proof of \cite[Lemma 2.7]{Boumal2018}. However, as they are also used later on, we present them as an independent lemma.
\begin{lemma}\label{lemma:propretr}
	Let $\mathcal{M} \subseteq \mathcal{E}$ be a smooth embedded manifold and $\retr: T\mathcal{M} \to \mathcal{M}$ a retraction on $\mathcal{M}$. Let $C \subseteq \mathcal{M}$ be compact. Then there exists a radius $r > 0$ and constants $\alpha, \beta \geq 0$ such that for all $x \in C$ the inequalities
	\begin{align}
		\|\retr_x(\xi_x) - x \| &\leq \alpha \| \xi_x\|,\label{eq:retrprop1}\\ 
		\|\retr_x(\xi_x) -x-\xi_x\| &\leq \beta \| \xi_x\|^2\label{eq:retrprop2}
	\end{align}
	hold for all $\xi_x \in \clball_r(0_x) \subseteq T_x\mathcal{M}$.
\end{lemma}
\begin{proof}
	Let $r>0$. By assumption, $\retr$ is defined and smooth on the set $K_0 := \{(x, \xi_x) \in T\mathcal{M}: x \in C, \|\xi_x\|\leq r\}$. Note that this set is compact by \cite[Lemma C.8]{levin2023} and hence (by continuity), $\tilde C := \retr(K_0)$ is compact, too. Now, denote

	\begin{equation}\label{eq:defkset}
		K:=\{(x, \xi_x) \in T\mathcal{M} \mid x \in \tilde C, \|\xi_x\| \leq r\},
	\end{equation}
	which again is compact by the same argument.

For the remaining part, we follow the proof of \cite[Lemma 2.7]{Boumal2018}. We first prove equation \eqref{eq:retrprop1}: Let $x \in C$, then for all $\xi_x \in \clball_r(0_x) \subseteq T_x \mathcal{M}$, we have
	\begin{align*}
		\|\retr_x(\xi_x)-x\| &\leq \int_0^1 \left\| \frac{\diff}{\diff t} \retr_x(t \xi_x) \right\| \diff t = \int_0^1 \|D\retr_x(t \xi_x)[\xi_x]\|\diff t \\
		& \leq \int_0^1 \max_{(z, \zeta_z) \in K} \|D\retr_z(\zeta_z)\| \|\xi_x\| \diff t = \max_{(z, \zeta_z) \in K} \|D\retr_z(\zeta_z)\|\|\xi_x\| =: \alpha \|\xi_x\|,
	\end{align*}
	where the maximum exists and is finite due to the smoothness of the retraction and the compactness of $K$.

	The second claim \eqref{eq:retrprop2} follows along similar lines: Let again $x \in C$, $\xi_x \in \clball_r(0_x)$, then
	\begin{align*}
		\|\retr_x(\xi_x) - x - \xi_x \| & = \|\retr_x(\xi_x) - \retr_x(0_x) - \xi_x \| \leq \int_0^1 \left\|\frac{\diff}{\diff t} \Big(\retr_x(t \xi_x) - t\xi_x\Big) \right\| \diff t \\
		& \leq \int_0^1 \|D \retr_x(t \xi_x)[\xi_x] - \xi_x\|\diff t \leq \int_0^1 \|D \retr_x(t \xi_x) - \mathrm{id}_{T_x \mathcal{M}}\| \|\xi_x\| \diff t\\
		&\leq \frac{1}{2} \max_{(z, \zeta_z) \in K} \| D^2 \retr_z (\zeta_z)\| \|\xi_x\|^2,
	\end{align*}
	where the last inequality follows from $\mathrm{id}_{T_x \mathcal{M}} = D\retr_x(0_x)$ and the following calculation:
	\begin{equation*}
		\|D\retr_x(t\xi_x) - \mathrm{id}_{T_x\mathcal{M}}\| \leq \int_0^1 \left\| \frac{\diff}{\diff s} D\retr_x(s t \xi_x) \right\| \diff s \leq \|t \xi_x\| \int_0^1 \| D^2 \retr_x(st\xi_x)\| \diff s.
	\end{equation*}
	Again, as $K$ is compact, the maximum above exists and hence \eqref{eq:retrprop2} follows with $\beta:= \frac{1}{2} \max_{(z, \zeta_z) \in K} \| D^2 \retr_z (\zeta_z)\|$.
\end{proof}

Finally, we prove a nonsmooth version of a corresponding result in \cite{Boumal2018}:
\begin{lemma}
	Let $\mathcal{M}$ be a (smooth embedded Riemannian) submanifold such that $\mathcal{M} \subseteq \mathcal{O} \subseteq \mathcal{E}$, where $\mathcal{O}$ is open. Further, assume that $\varphi$ is upper-$\mathcal{C}^2$ on $\mathcal{O}$ and let $\retr$ be a retraction on $\mathcal{M}$. Then Assumption~\ref{descentas} is satisfied.
\end{lemma}
\begin{proof}
	The local boundedness of $\partial_\mathcal{M} \varphi$ is a direct consequence of Lemma~\ref{lem:projsubdiff}.

	By Lemma~\ref{lemma:propretr}, there exists some $r > 0$ such that \eqref{eq:retrprop1} and \eqref{eq:retrprop2} hold (also in an open neighborhood). Further, by Proposition~\ref{Prop:CharUpperC2}, we have that for every point $\bar x$ in $\mathcal{O}$, the descent property \eqref{eq:descentproperty} holds for all $x, y$ in an open neighborhood $V$. As $\retr$ is continuous, $\retr^{-1}(V)$ is an open neighborhood of $(\bar x, 0_{\bar x}) \in T \mathcal{M}$. Hence (see e.g.\ \cite{doCarmo1992}), there exists an open neighborhood $U \subseteq \mathcal{M}$ of $\bar x$ and $r>0$ such that
	\begin{equation*}
		\mathcal{U} := \{(x, \xi_x) \in T\mathcal{M} \mid x \in U, \|\xi_x\| < r\} \subseteq \retr^{-1}(V),
	\end{equation*}
	and $\mathcal{U}$ is an open neighborhood of $(\bar x, 0_{\bar x})$ in $T\mathcal{M}$. Thus, by taking $x$ from a possibly smaller open neighborhood, we can choose $r$ small enough such that both properties above hold for all $\xi_x \in \clball_r(x)$, with $y = \retr_x(\xi_x)$ in \eqref{eq:descentproperty}, that is
	\begin{equation}\label{eq:descentproperty2}
		\varphi(\retr_x(\xi_x)) \leq \varphi(x) + \langle w, \retr_x(\xi_x) - x \rangle + \kappa \|\retr_x(\xi_x) - x\|^2
	\end{equation}
	holds for all $w \in \partial \varphi(x)$. 

	Now, let $w_\mathcal{M} \in \partial_\mathcal{M} \varphi(x)$. As in the proof of Lemma~\ref{lem:projsubdiff}, we find $w \in \partial \varphi(x)$ and $n \in N_\mathcal{M}(x)$, such that $w_\mathcal{M} = w+n$. It now follows from the orthogonality $n \perp \xi_x$ that
	\begin{align*}
		\langle w, \retr_x(\xi_x) - x\rangle &= \langle w, \xi_x + \retr_x(\xi_x) - x - \xi_x\rangle\\
			&= \langle w_\mathcal{M}, \xi_x\rangle + \langle w, \retr_x(\xi_x) - x - \xi_x\rangle.
	\end{align*}
	Thus, by \eqref{eq:descentproperty2}, we have
	\begin{align*}
		\varphi(\retr_x(\xi_x)) &\leq \varphi(x) + \langle w_\mathcal{M}, \xi_x \rangle + \langle w, \retr_x(\xi_x) - x - \xi_x\rangle + \kappa \|\retr_x(\xi_x) - x\|^2\\
			&\leq \varphi(x) + \langle w_\mathcal{M}, \xi_x \rangle + \|w\| \|\retr_x(\xi_x) - x - \xi_x\| + \kappa \|\retr_x(\xi_x) - x\|^2.
	\end{align*}
	Now, the set $\partial \varphi(x)$ is bounded by assumption, thus there exists $G\geq 0$ (independent of the choice of $w_\mathcal{M}$), such that $\|w\|\leq G$ in the equation above.

	Combining \eqref{eq:retrprop1} and \eqref{eq:retrprop2} with the calculation above, we deduce
	\begin{equation*}
		\varphi(\retr_x(\xi_x)) \leq \varphi(x) + \langle w_\mathcal{M}, \xi_x \rangle + (G\beta + \kappa \alpha^2) \|\xi_x\|^2.
	\end{equation*}
	As $w_\mathcal{M} \in \partial_\mathcal{M} \varphi(x)$ and $\xi_x \in \clball_r(0_x)$ were arbitrary, the claim follows.
\end{proof}

\section{Descent Method and Global Convergence}\label{Sec:Algorithm}

Before discussing our method and its convergence properties, let us first provide a motivation for the proposed algorithm. The standard method for the unconstrained minimization of a continuously differentiable 
function $ \varphi $ is based on the iteration
\begin{equation}\label{eq:opt-iteration}
	\xkp := \xk + \tau_k \dk, 
\end{equation}
with a search direction $ \dk \in \mathbb{R}^n $ satisfying the descent
property $ \langle \nabla \varphi (\xk), \dk \rangle < 0 $ and a 
stepsize $ \tau_k $ such that a suitable line search criterion holds
like the Armijo rule 
\begin{equation}\label{eq:Armijo}
	\varphi (\xk + \tau_k \dk) \leq \varphi (\xk) + \sigma \tau_k
	\langle \nabla \varphi (\xk), \dk \rangle
\end{equation}
for some constant $ \sigma \in (0,1) $. The descent method considered 
in this section is a direct generalization of this approach to the class
optimization problems on some manifold $\mathcal{M}$ with objective functions $ \varphi $ which are upper-$\mathcal{C}^2$ (on an open superset of $\mathcal{M}$ in the embedding space). 
In the Euclidean setting, which was considered in \cite{aragon2025}, the
monotone version is based on the iteration \eqref{eq:opt-iteration}
with some search direction $ d^k $ satisfying the descent-like property
$ \langle \wk, \dk \rangle < 0 $ for an arbitrary element $ \wk \in 
\partial \varphi (\xk) $. The counterpart of the (monotone) Armijo rule 
\eqref{eq:Armijo} reads
\begin{equation*}
	\varphi (\xk + \tau_k \dk) \leq \varphi (\xk) + \sigma \tau_k
	\langle \wk, \dk \rangle
\end{equation*}
for some $ \sigma \in (0,1) $. 

Linesearch methods have also become popular for problems constrained to smooth manifolds, see e.g.\ \cite{Absil2008,Boumal2018,Oviedo2022}. The generalization from the Euclidean setting is straight-forward, by simply replacing the canonical retraction $x \mapsto x+\xi$ in the Euclidean setting by a retraction defined for the respective manifold. Under suitable assumptions on $\dk$, the linesearch criterion is given by 
\begin{equation*}
	\varphi(\retr_{\xk} (\dk))\leq \varphi(\xk) + \sigma \tau_k \langle \grad \varphi(\xk), \dk\rangle,
\end{equation*}
when $\varphi$ is a differentiable function. In our nonsmooth case, we will consider the following condition, where we choose some $w_\mathcal{M}^k \in \partial_\mathcal{M} \varphi(\xk)$ and a direction $\dk$:
\begin{equation*}
	\varphi(\retr_{\xk} (\dk))\leq \varphi(\xk) + \sigma \tau_k \langle w^k_\mathcal{M}, \dk\rangle.
\end{equation*}

In order to prove global convergence results, we need to impose some
conditions on the quality of the descent direction $ d^k $ and the 
underlying choice of the subgradients $ w_\mathcal{M}^k $ to ensure that $d^k$ remains related to the subgradients. The following conditions are similar to those already specified in \cite{aragon2025} for upper-$\mathcal{C}^2$ functions in the Euclidean case and also appear already in \cite{ZhangHager2004}. Note that these conditions are also standard for linesearch methods in smooth manifold optimization, see \cite{Boumal2018, Oviedo2022, Qian2024}, where a number of equivalent conditions are used.

\begin{assumption}\label{genas}
Assume:
\begin{itemize}
	\item[(a)] There exists a constant $a > 0$ such that 
		\begin{equation}\label{eq:wdbound1}
			\langle \wk_\mathcal{M}, \dk\rangle \leq - a \|\dk\|^2, \text{ for all } k \in \mathbb{N}.
		\end{equation}
	\item[(b)] There exists a constant $b >0$ such that
		\begin{equation}\label{eq:wdbound2}
			\|\wk_\mathcal{M}\| \leq b \|\dk\|, \text{ for all } k \in \mathbb{N}.
		\end{equation}
\end{itemize}
\end{assumption}
Note that for $\dk \neq 0$, we obtain from Cauchy-Schwarz that
\begin{equation}\label{eq:wdbound3}
	a \| \dk \| \leq \|\wk_\mathcal{M}\|,
\end{equation}
and also
\begin{equation}\label{eq:wdbound4}
	\langle -\wk_\mathcal{M}, \dk \rangle \geq a \| \dk\|^2 \geq \frac{a}{b} \|\dk\|\|\wk_\mathcal{M}\|.
\end{equation}

The following result (see \cite[Proposition 4.3]{aragon2025} for a similar result for upper-$\mathcal{C}^2$ functions in the Euclidean setting and \cite[Lemma 2.10]{Boumal2018} for the smooth case on a manifold) shows that the Armijo-type condition for our linesearch is satisfied for all $ \tau_k > 0 $ sufficiently small.

\begin{proposition}\label{prop:welldeflinesearch}
	Assume that $\varphi: \mathcal{M} \to \R$ satisfies Assumption~\ref{descentas}. Given any $\sigma \in (0, 1)$, $\xk \in \mathcal{M}$, $\dk \in T_{\xk}\mathcal{M}$ and $\wk_\mathcal{M} \in \partial_\mathcal{M} \varphi(\xk)$ such that $\wk_\mathcal{M}$ and $\dk$ satisfy Assumption~\ref{genas}, there exists a $t > 0$ such that
\begin{equation}\label{eq:welldef}
	\varphi\big(\retr_{\xk}(\tau \dk)\big) \leq \varphi(\xk) + \sigma \tau \langle \wk_\mathcal{M}, \dk\rangle, \text{ for all } \tau \in (0, t).
\end{equation}
\end{proposition}
\begin{proof}
	By Assumption~\ref{descentas} there exists some $r>0$ such that \eqref{eq:descentpropm} holds, i.e.
	\begin{equation}\label{eq:descentpropproof1}
		\varphi\big(\retr_{\xk}(\dk)\big) - \big(\varphi(\xk) + \langle \wk_\mathcal{M}, \dk \rangle\big) \leq \tilde \kappa \|\dk\|^2
	\end{equation}
	for all $\dk \in \clball_r(0_{\xk}) \subset T_{\xk} \mathcal{M}$ and all $\wk_\mathcal{M} \in \partial_\mathcal{M} \varphi(\xk)$.

	For $\dk = 0$, \eqref{eq:welldef} is obviously true (for all $t>0$). Otherwise, we now show that \eqref{eq:welldef} holds for all $\tau \in (0, t)$ with $t = \min\left\{\frac{r}{\|\dk\|}, \frac{a^2(1-\sigma)}{b\tilde \kappa}\right\}$. By applying first \eqref{eq:wdbound3} and then \eqref{eq:wdbound4}, we obtain
	\begin{equation*}
		0 < \tau < t \leq \frac{a^2(1-\sigma)}{b \tilde \kappa} \leq \frac{a (1-\sigma)\|\wk_\mathcal{M}\|}{b \tilde \kappa \| \dk \|}\leq \frac{(1-\sigma) \langle - \wk_\mathcal{M}, \dk \rangle}{\tilde \kappa \|\dk\|^2}.
	\end{equation*}
	Consequently, one has for all $\tau \in (0, t)$
	\begin{equation*}
		\tilde \kappa \tau^2 \|\dk\|^2 \leq (1-\sigma) \tau \langle -\wk_\mathcal{M}, \dk \rangle 
	\end{equation*}
	and therefore in combination with \eqref{eq:descentpropproof1} with $\dk$ replaced by $\tau \dk$ it holds that
	\begin{equation*}
		- \sigma \tau \langle \wk_\mathcal{M}, \dk \rangle \leq - \tau \langle \wk_\mathcal{M}, \dk \rangle - \tilde \kappa \tau^2 \|\dk\|^2 \leq \varphi(\xk) - \varphi(\retr_{\xk} (\tau \dk)).
	\end{equation*}
	Finally, by rearranging terms, we obtain the claimed inequality for all $\tau \in (0, t)$:
	\begin{equation*}
		\varphi(\retr_{\xk}(\tau \dk)) \leq \varphi(\xk) + \sigma \tau \langle \wk_\mathcal{M}, \dk \rangle.
	\end{equation*}
	This completes the proof.
\end{proof}
We next turn to a nonmonotone version of the previous iteration. 
It is clear from \eqref{eq:welldef} and Assumption~\ref{genas} that an algorithm with the iteration $\xkp = \retr_{\xk}(\tauk \dk)$, where $\dk$ is such the conditions \eqref{eq:wdbound1} and \eqref{eq:wdbound2} hold and $\tauk$ is determined by a backtracking linesearch with termination criterion as in \eqref{eq:welldef}, will produce a monotonically decreasing sequence of function values $\{\varphi(\xk)\}_{k \in \mathbb{N}}$.

Of course, \eqref{eq:welldef} will also hold if $\varphi(\xk)$ is replaced by an arbitrary upper bound. Therefore, suppose that we have a reference value $\Rvalue_k \geq \varphi(\xk)$. Then Proposition \ref{prop:welldeflinesearch} guarantees that a backtracking linesearch with termination criterion
\begin{equation}\label{eq:linesearchtermination}
	\varphi(\retr_{\xk}(\tauk \dk)) \leq \Rvalue_k + \sigma \tauk \langle \wk_\mathcal{M}, \dk \rangle
\end{equation}
terminates after a finite number of iterations. Hence, provided $\xk$ is not already a stationary point, our linesearch method is well-defined.

Further, the condition \eqref{eq:linesearchtermination} might already be satisfied for a larger choice
of $ \tau_k $ in comparison to the monotone version. This results in possibly larger steps, which is the
reason why nonmonotone methods may outperform their monotone
counterparts in practical applications. 

In order to obtain
suitable (global) convergence results, the reference value $ \Rvalue_k $
has to be chosen in a careful way. One popular choice is due to
Grippo et al.\ \cite{GrippoLamparielloLucidi1986}, 
where $ \Rvalue_k := \max \{ \varphi(\xj) \mid j = k, k-1, \ldots, k - m_k \} $ for some given $ m_k \in \N $. We call this
strategy the \emph{max-rule} since $ \Rvalue_k $ is defined as the maximum
function value over the last few iterates, say, the last ten points. In the Euclidean setting, this choice of reference values was already studied in \cite{aragon2025}, where stationarity of accumulation points for arbitrary upper-$\mathcal{C}^2$ functions was derived.

In our Algorithm~\ref{Alg:NonmonotoneSubgradient}, however, we use
the technique introduced by Zhang and Hager \cite{ZhangHager2004}, 
where $ \Rvalue_{k+1} $
is computed as a convex combination of the previous reference value
$ \Rvalue_k $ and the new function value $ \varphi (\xkp) $. We 
therefore call this the \emph{mean-rule}. Equation~\eqref{eq:descentcondition} in the upcoming result Lemma \ref{lemma:genproperties} shows that $\Rvalue_k$ chosen by the mean-rule is indeed an upper bound for $\varphi(\xk)$ in our setting. 
In the case where $\varphi$ is a differentiable function and hence $\wk_\mathcal{M} = \grad \varphi(\xk)$,
the nonmonotone linesearch with mean-rule is well-established in manifold optimization and was
already studied in \cite{Oviedo2022} and \cite{Qian2024}. Let us note, however, that even in the differentiable case, 
the assumptions are usually more restrictive. Most notably, Assumption~\ref{descentas} is weaker than the assumption 
that $\varphi$ has a globally Lipschitz continuous gradient on an open superset of the embedded manifold $\mathcal{M}$.
Also, unlike previous results, in our case, $\mathcal{M}$ does not need to be compact. The algorithmic details are presented in Algorithm~\ref{Alg:NonmonotoneSubgradient}.

Assumption~\ref{genas} allows a wide variety of options regarding the choice
of the search direction $ \dk $. For example, let $ \wk_\mathcal{M} \in \partial_\mathcal{M} 
\varphi (\xk) $ be arbitrarily given. If $ \wk_\mathcal{M} = 0 $, then $ \xk $
is already a stationary point, and we terminate the iteration. Otherwise,
we have $ \wk_\mathcal{M} \neq 0 $, in which case the gradient-type direction
$ \dk := - \wk_\mathcal{M} $ satisfies the two conditions (a) and (b) from 
Assumption~\ref{genas} with $ a:= b:= 1 $. However, in the Euclidean setting, 
we can also take well-known quasi-Newton or limited-memory quasi-Newton directions
like $ \dk := -H_k \wk_\mathcal{M} $ for a bounded and uniformly positive definite 
sequence of (limited-memory) matrices $ \{ H_k \} $. Then both conditions
from Assumption~\ref{genas} are still satisfied. In the general setting, where $\mathcal{M}$ itself is 
not a Euclidean space, there exist some generalizations of quasi-Newton methods to Riemannian manifolds, see e.g.\ \cite{Qi2010, Huang2016}.
However, they often involve vector transports between tangent spaces, which may be computationally expensive \cite{Godaz2021}.

In the Euclidean setting, our main convergence result reduces to
essentially the same as the one from \cite{aragon2025}. However, 
as we use the mean rule as opposed to the max rule, the convergence analysis turns out to be much 
simpler.

\begin{algorithm}[Nonmonotone Subgradient Method on Manifolds]\leavevmode
	\label{Alg:NonmonotoneSubgradient}
\begin{algorithmic}[1]
	\Require $x^0 \in \mathcal{M}$, a retraction $\retr$ on $T\mathcal{M}$, $0<\tau_\mathrm{min} \leq \taumax < \infty$, $\sigma, \beta \in (0,1)$, $\pmin \in (0, 1]$.
	\State Set $\Rvalue_0 := \varphi(x^0)$.
	\For{$k = 0, 1, 2, \dots$}
	\State Choose $\wk_\mathcal{M} \in \partial_\mathcal{M} \varphi(\xk)$; \If{$\wk = 0$} \State STOP and return $\xk$.\EndIf
	\State Choose $\dk \in T_{\xk}\mathcal{M} \setminus \{0_{\xk}\}$ such that \eqref{eq:wdbound1} and \eqref{eq:wdbound2} hold.
	\State Choose $\tau_k \in [\taumin, \taumax]$.
	\While{$\varphi(\retr_{\xk}(\tauk \dk)) > \Rvalue_k + \sigma \tauk \langle \wk_\mathcal{M}, \dk \rangle$}
	\State $\tauk = \beta \tauk$
	\EndWhile
	\State Set $\xkp := \retr_{\xk}(\tauk \dk)$.
	\State Choose $p_{k+1} \in [\pmin, 1]$ and set $\Rvalue_{k+1} := (1-p_{k+1})\Rvalue_k + p_{k+1} \varphi(\xkp)$.
	\EndFor
\end{algorithmic}
\end{algorithm}

Note that the sequence $\{\xk\}$ generated by Algorithm~\ref{Alg:NonmonotoneSubgradient} satisfies $\xkp \neq \xk$ for all $k$ as otherwise, one would have $\xk = \xkp = \retr_{\xk}(\tauk \dk)$ and hence by the property that $\retr_{\xk}$ is locally invertible, we have $\tauk \dk = 0_{\xk}$ contradicting our choice of $\dk$ and $\tauk > 0$. Now, we first collect some more properties of the sequence $\{\xk\}$, which are similar to those obtained for a nonmonotone proximal gradient method in \cite{DeMarchi2023}.

\begin{lemma}\label{lemma:genproperties}
	Let Assumption~\ref{descentas} and Assumption~\ref{genas} be satisfied. Assume that $\inf_{x \in \mathcal{M}} \varphi(x) > - \infty$. Then for all $x^0 \in \mathcal{M}$, Algorithm~\ref{Alg:NonmonotoneSubgradient} either stops at some stationary point after a finite number of iterations, or the sequence $\{\xk\}_{k \in \mathbb{N}}$ satisfies the following properties:
\begin{itemize}
	\item[(a)] For all $k \in \mathbb{N}$ it holds that
		\begin{equation}\label{eq:descentcondition}
			\varphi(\xkp) + (1- p_{k+1}) \delta_k \leq \Rvalue_{k+1} \leq \Rvalue_k - p_{k+1} \delta_k,
		\end{equation}
		where 
		\begin{equation}
			\delta_k := - \sigma \tauk \langle \wk_\mathcal{M}, \dk \rangle \geq \sigma a\tauk \|\dk\|^2 \geq 0.
		\end{equation}
	\item [(b)] The sequence $\{ \Rvalue_k\}$ is monotonically decreasing.
	\item [(c)] Both $\{ \Rvalue_k \}$ and $\{ \varphi(\xk)\}$ converge to some value $\varphi^* \in \mathbb{R}$, i.e., both sequences converge and
	have the same limit.
	\item [(d)] $\tauk \dk \to 0$ for $k \to \infty$.
\end{itemize}
\end{lemma}

\begin{proof}
The first inequality of part (a) follows from the definition of $\Rvalue_{k+1}$ and the linesearch termination criterion in \eqref{eq:linesearchtermination}:
\begin{align*}
	\Rvalue_{k+1} &= (1-p_{k+1}) \Rvalue_k + p_{k+1} \varphi(\xkp) \\
	&\geq (1-p_{k+1}) \big(\varphi(\xkp) - \sigma \tauk \langle \wk_\mathcal{M} , \dk \rangle\big) + p_{k+1} \varphi(\xkp)\\
	&= \varphi(\xkp) - (1-p_{k+1})\sigma \tauk \langle \wk_\mathcal{M}, \dk \rangle.
\end{align*}
Similarly, we verify the second inequality:
\begin{align*}
	\Rvalue_{k+1} &= (1- p_{k+1}) \Rvalue_k + p_{k+1} \varphi(\xkp) \\
	&\leq (1-p_{k+1}) \Rvalue_k + p_{k+1} ( \Rvalue_k + \sigma \tau_k \langle \wk_\mathcal{M}, \dk \rangle)\\
	&= \Rvalue_k + p_{k+1} \sigma \tauk \langle \wk_\mathcal{M}, \dk \rangle.
\end{align*}
The second assertion directly follows from part (a), and we obtain for all $k$ that
\begin{equation*}
	\varphi(\xk) \leq \Rvalue_k \leq \dots \leq \Rvalue_0 = \varphi(x^0).
\end{equation*}
The last equation also shows that as $\varphi$ is bounded from below, the sequence of reference values $\{\Rvalue_k\}$ is convergent to some limit $\varphi^*$. Now, as $p_k \geq \pmin$ and
\begin{equation*}
	\varphi(\xk) = \frac{1}{p_k} \big( \Rvalue_k - (1-p_k) \Rvalue_{k-1}
	\big) = \frac{1}{p_k} (\Rvalue_k - \Rvalue_{k-1}) + \Rvalue_{k-1}
\end{equation*}
due to the update of $ \Rvalue_k $,
the (usually nonmonotone) convergence of $\{\varphi(\xk)\}$ to the same limit $\varphi^*$ follows.

Statement (d) is now a consequence of part (a)
and the following telescoping argument:
\begin{equation}\label{eq:telescoperval}
	\infty > \Rvalue_0 - \varphi^* \geq \sum_{j=1}^k 
	\big( \Rvalue_{j-1} - \Rvalue_j \big) \geq \sum_{j=1}^k p_j \delta_{j-1} \geq \sum_{j=1}^k \pmin \sigma a \frac{\tau_{j-1}^2}{\taumax} \|d^{j-1}\|^2,
\end{equation}
where the upper bound $\Rvalue_0 - \varphi^*$ is independent of $k$ and holds for all $k \in \mathbb{N}$.
\end{proof}
Note that \eqref{eq:telescoperval} directly implies the following estimate similar to \cite[Proposition 4.4]{aragon2025}:

\begin{proposition}\label{prop:convest}
There exists a constant $c > 0$ such that
\begin{equation}
	\min_{0 \leq j \leq k} \tauk \|\dk\| \leq \frac{c\sqrt{\Rvalue_0 - \varphi^*}}{\sqrt{k+1}}, \text{ for all } k \in \mathbb{N}.
\end{equation}
\end{proposition}
\begin{proof}
By \eqref{eq:telescoperval} we obtain
\begin{equation*}
	\min_{0 \leq j \leq k} \tauk^2 \|\dk\|^2 \leq \frac{1}{k+1} \frac{\taumax}{\pmin \sigma a}\sum_{j=0}^{k} \pmin \frac{\sigma a}{\taumax} \tau_j^2 \|d^j\|^2
	\leq \frac{1}{k+1} \frac{\taumax}{\pmin \sigma a} \big(\Rvalue_0 - \varphi^* \big).
\end{equation*}
Taking the square root gives the claim with $c:= \sqrt{\frac{\taumax}{\pmin \sigma a}}$.
\end{proof}

Next, we establish subsequential convergence to stationary points for the sequence $\{\xk\}$ generated by Algorithm~\ref{Alg:NonmonotoneSubgradient}. 

\begin{theorem}\label{thm:subsequenceconvergence}
Assume that Assumptions~\ref{descentas} and \ref{genas} hold and that $\varphi$ satisfies $\inf_{x \in \mathcal{M}} \varphi(x) > - \infty$. Then, given any $x^0 \in \mathcal{M}$, either Algorithm~\ref{Alg:NonmonotoneSubgradient} stops at a stationary point after a finite number of iterations, or the following assertions hold for the generated sequence $\{\xk\}$:
\begin{enumerate}
	\item Suppose that $\{\xk\}_{k \in K}$ is a bounded subsequence of $\{\xk\}$, then the corresponding stepsizes are uniformly bounded away from zero, i.e.\ $\inf_{k \in K} \tau_k > 0$. Further, it holds that $\| \wk_\mathcal{M} \| \to_K 0$ and $\| \dk \| \to_K 0$. 
	\item Any accumulation point of $\{\xk\}_{k \in \mathbb{N}}$ is a stationary point.
	\item Let $x^*$ be an accumulation point of $\{\xk\}$. Then the entire sequence $ \{ \varphi (\xk) \} $ converges to $\varphi (x^*) $ and the sequence $ \{ \Rvalue_k \} $ converges monotonically to $ \varphi (x^*) $.
\end{enumerate}
\end{theorem}

\begin{proof}
The crucial part is the claim that $\inf \tau_k > 0$ on bounded subsequences. Assume, by contradiction, that $\inf_{k \in K} \tau_k = 0$, where $\{\xk\}_{k \in K}$ is bounded. By taking subsequences, we may assume that $\tau_k \to_K 0$, and by boundedness of $\{\xk\}_K$, we assume that there exists some $\bar x$ such that $\xk \to_K \bar x$.

By Lemma~\ref{lem:projsubdiff}, we have that for all $\wk_\mathcal{M}$ there exists $\wk \in \partial \varphi(\xk), k \in K$, such that $\wk_\mathcal{M} = \proj_{T_x\mathcal{M}} (\wk)$, it follows that as $\{\wk\}_{k \in K}$ is bounded, $\{\wk_\mathcal{M}\}_K$ is bounded, too. That is, by taking again a suitable subsequence, we have $\wk_\mathcal{M} \to_K \bar w_\mathcal{M}$ for some $\bar w_\mathcal{M}$. From \eqref{eq:wdbound3}, we have $\| \wk \| \geq a \| \dk \|$, which implies that $\{ \dk\}_K$ is also bounded and hence, again on a suitable subsequence, we may assume $\dk \to_K \bar d$ for some $\bar d \in \X$. As $\tau_k \to_K 0$, we assume $\tauk < \taumin$ 
for all $k \in K$ sufficiently large. Denote by $\tauhk$ the previous 
trial stepsize, i.e., $\tauhk = \beta^{-1} \tauk$ and $\xhkp := \retr_{\xk}(\tauhk \dk)$. That is, the inner loop does not terminate with the stepsize $\tauhk$, therefore
\begin{equation}\label{eq:notermcrit}
	\varphi(\xhkp) > \Rvalue_k + \sigma \tauhk \langle \wk_\mathcal{M}, \dk \rangle.
\end{equation}
Since $ \tauhk \| \dk\| \to_K 0$, it follows by continuity of the retraction that $\xhkp \to_K \bar x$. By our Assumption~\ref{descentas} there exists a constant $\hat \kappa > 0$ such that, for all $k \in K$ sufficiently 
large, we have
\begin{equation}\label{eq:charctwobarx}
	\varphi(\xhkp) \leq \varphi(\xk) + \tauhk \langle \wk_\mathcal{M}, \dk \rangle + \hat \kappa \tauhk^2 \| \dk \|^2.
\end{equation}
A combination of first \eqref{eq:notermcrit}, the fact that $\Rvalue_k \geq \varphi(\xk)$, followed by \eqref{eq:charctwobarx} and \eqref{eq:wdbound1} gives
\begin{align*}
	\sigma \tauhk \langle \wk_\mathcal{M}, \dk \rangle & < \varphi(\xhkp) - \Rvalue_k  \leq \varphi(\xhkp) - \varphi(\xk)\\
		& \leq \tauhk \langle \wk_\mathcal{M}, \dk \rangle + \hat \kappa \tauhk^2 \|\dk\|^2
		 \leq \tauhk  \left( 1 - \frac{\hat \kappa}{a} \tauhk\right) \langle \wk_\mathcal{M}, \dk \rangle.
\end{align*}
Therefore, as $\langle \wk_\mathcal{M}, \dk \rangle < 0$, it follows that $\sigma > 1- \frac{\hat \kappa}{a} \tauhk$.
As $\tauk \to_K 0$ it follows that $\tauhk \to_K 0$. Thus, the above inequality contradicts the fact that $\sigma \in (0,1)$.

The claim concerning the convergence of $\{\dk\}_K$ is now a direct consequence of Lemma~\ref{lemma:genproperties} (d) and since $\| \wk_\mathcal{M} \| \leq b \|\dk \|$ by Assumption~\ref{genas}, it also follows that $\|\wk\| \to_K 0$.

Now, we are ready to verify the second claim. Let $x^*$ be an accumulation point of $\{\xk\}_{k \in \mathbb{N}}$, hence, there exists a subsequence $\{ \xk \}_{k \in K}$ converging to $x^*$. By the first part, we have $\wk_\mathcal{M} \to_K 0$, which, by the so-called robustness property of the projectional subdifferential implies that 
$0 \in \partial_\mathcal{M} \varphi(x^*)$ as $\wk_\mathcal{M} \in \partial_\mathcal{M} \varphi (\xk)$ for all $k \in K$.

By Lemma~\ref{lemma:genproperties} we already know that both $\{\varphi(\xk)\}$ and $\{\Rvalue_k\}$ converge (in the latter case monotonically) and their limits agree. However, on a subsequence, we have $\varphi(\xk) \to_K \varphi(x^*)$ by continuity of $\varphi$. Thus, the last claim concerning the convergence of $\{\varphi(\xk)\}$ and $\{\Rvalue_k\}$ follows.
\end{proof}

\section{Convergence under the Kurdyka–Łojasiewicz Property}\label{Sec:KL}

The key idea to obtain global and rate-of-convergence results for the nonmonotone subgradient method under the KL property is to exploit that the stepsizes along bounded subsequences remain bounded. This was shown in Theorem~\ref{thm:subsequenceconvergence} and allows us to weaken the global assumptions in \cite{Qian2024} (for the differentiable case) to merely local ones in the spirit of \cite{kanzow2025}. Therefore, we restrict the discussion here to highlight the main differences to the previous works \cite{kanzow2025} and \cite{Qian2024}. We omit technical details and only discuss a sketch of proof. The full proofs can be found in the appendix.

For the remaining part, we assume that the Assumptions~\ref{descentas} and \ref{genas} hold. Further, assume that our objective function 
\begin{equation*}
	\Phi := \varphi + \delta_\mathcal{M}
\end{equation*}
satisfies the KL property at a given accumulation point $x^* \in \mathcal{M}$. Let $\eta > 0$ be the corresponding constant and $\chi$ the associated desingularization function from Definition~\ref{def:klproperty}, and denote by $\{\xk \}_{k \in K}$ a subsequence converging to $x^*$. 

Let $C \subseteq \mathcal{M}$ be a compact subset with nonempty interior relative to $\mathcal{M}$, such that $x^* \in \mathrm{int} C$ and $\varphi(x) \leq \Rvalue_0$ for all $x \in C$. Such a set always exists unless we are in the trivial case $x_0 = x^*$. By Theorem~\ref{thm:subsequenceconvergence}, there exists a constant $\taug_C > 0$ 
such that 
\begin{equation}\label{eq:taucbounded}
   \tauk \geq \taug_C \quad \text{for all } k
   \text{ with } \xk \in C.
\end{equation}
We denote the constants from Lemma~\ref{lemma:propretr} by $\alpha, \beta$ respectively and note that these can be chosen as global constants on the set $C$.

Following mostly \cite{Qian2024}, we also introduce
the subsequent notation: 
\begin{itemize}
	\item $m := \min \big\{ l \in \N \, \big| \, (1-\sqrt{1-p_{\min}}) \sqrt{l} \geq (1+\sqrt{1-p_{\min}}) \big\} $,
	\item $l(k) := k + m - 1$, 
	\item $\Xi_{k-1} := \sqrt{\Rvalue_{k-1} - \Rvalue_{k}}$ for $k \in \N $
	   and
	\item $\Delta_{i, j} := \chi \big(\Rvalue_i- \varphi(x^*) \big) - 
	   \chi \big( \Rvalue_j - \varphi(x^*) \big)$.
\end{itemize}
The index $m$ is uniquely defined as the 
left-hand side of the inequality in its definition eventually becomes larger than
the constant on the right-hand side. Moreover, 
the difference $ l(k)- k = m-1 $ is a constant number for all $ k \in \N $.
It is thus guaranteed later that certain
sums are always taken over a finite (fixed) number of terms only. 

Since $\{\Rvalue_k\}$ is a monotone sequence, we have that $\Xi_{k-1} \geq 0$ holds for all $ k \in \N $.
In combination with the monotonicity of $\chi$, it follows that $\Delta_{i,j} \geq 0$ for all $j \geq i$.

Let us also introduce the two index sets
\begin{equation*}
	K_1 := \big\{ k \in \N \, \big| \, \varphi(\xk) \leq \Rvalue_{k+m}
	\big\} \text{\ and\ } K_2 := \big\{k \in \N \, \big| \, \varphi(\xk) > \Rvalue_{k+m} \big\}
\end{equation*}
depending on the previously introduced number $ m $.

From Lemma~\ref{lemma:genproperties}, we have $\tauk \dk \to 0$ and 
\begin{equation*}
	\pmin \sigma a \tauk \|\dk\|^2 \leq \Rvalue_k - \Rvalue_{k+1}.
\end{equation*}
Assuming $x^k \in C$, we can deduce the following inequality
\begin{equation*}
	\| \xkp - \xk \|^2  \leq \alpha^2 \tauk^2 \|\dk\|^2 
		 \leq \frac{\alpha^2 \tauk}{\pmin \sigma a} (\Rvalue_k - \Rvalue_{k+1}) \leq \frac{\alpha^2 \taumax}{\pmin \sigma a} \Xi_k^2,
\end{equation*}
and hence
\begin{equation}\label{eq:normxi}
	e \|\xkp -\xk\| \leq \Xi_k,
\end{equation}
where $e$ denotes the constant $e = \frac{1}{\alpha} \sqrt{\frac{\pmin \sigma a}{\taumax}}$.

The following result is clear, as the individual summands can be made arbitrarily small (recall again that the difference $ l(k) - k = m - 1 $ is
a constant).

\begin{lemma}\label{Lem:alpha}
Define the constant $ \hat{c} := \frac{b \alpha}{\taug_C} \sqrt{\frac{\taumax \pmin}{\sigma a}}$, with $\taug_C$ as in \eqref{eq:taucbounded}. Then there
exists a sufficiently large index $k_0-1 \in K$ such that
\begin{equation}
	\vartheta := \|x^{k_0-1} - x^*\| + \frac{1}{e} \sum_{j = k_0}^{l(k_0)} (3 \Xi_{j-1} + \Xi_j) + \frac{2 \hat c}{e} \sum_{j = k_0}^{l(k_0)} \chi \big( \Rvalue_j - \varphi(x^*) \big)
\end{equation}
	satisfies $\clball_\vartheta(x^*) \subset U$, where $U$ is the neighborhood of $x^*$ from the KL property in Definition~\ref{def:klproperty} and $\clball_\vartheta(x^*) \cap \mathcal{M} \subset C$.
\end{lemma}

We next provide an upper bound for the distance of the current point $ \xk $ to
stationarity, measured by the expression $ \dist \big( 0, \partial \Phi(\xk) \big) $, see also \cite{Qian2024}.

\begin{lemma}\label{lemma:bs}
Under the conditions specified above, we have
\begin{equation}\label{distbound}
	\dist \big( 0, \partial \Phi(\xk) \big) 
	\leq \frac{2 b}{\taug_C} \| \xkp - \xk\|,
\end{equation}
for all sufficiently large $k$ with $\xk \in \clball_\vartheta(x^*)$.
\end{lemma}

\begin{proof}
By Theorem~\ref{thm:subsequenceconvergence} and noting that $\clball_\vartheta(x^*) \cap \mathcal{M} \subseteq C$, we obtain that $\tauk \geq \taug_C$ for all $k$ with $\xk \in \clball_\vartheta(x^*)$. Therefore, by Lemma~\ref{lemma:propretr}, we have
\begin{equation*}
	\|\xkp - (\xk + \tauk \dk)\| = \|\retr_{\xk}(\tauk \dk) - (\xk + \tauk \dk)\| = o(\| \tauk \dk\|).
\end{equation*}
Now the claim follows by Assumption~\ref{genas} as for all $k$ large enough
\begin{align*}
	\| \xkp - \xk \| &\geq \tauk \| \dk \| - \| \retr_{\xk} (\tauk \dk) - (\xk + \tauk \dk)\|\\
		& \geq \frac{1}{2} \taug_C \| \dk \| 
		 \geq \frac{1}{2 b} \taug_C \| \wk_\mathcal{M}\| 
		 \geq \frac{1}{2b} \taug_C \dist\big(0, \partial \Phi(\xk)\big),
\end{align*}
where the last inequality is due to \eqref{eq:projsubdiffchar}.
\end{proof}

An application of the corresponding result in \cite{Qian2024} to the setting considered here, gives us the following technical result:

\begin{lemma}\label{lemma:lemmaxi}
For all sufficiently large $k \in \mathbb{N}$ with $\Rvalue_k < \varphi(x^*) + \eta$ and $\xk \in \clball_\vartheta(x^*)$, the following inequality holds:
\begin{equation}\label{eq:lemmaxi}
	\frac{1-\sqrt{1-p_{\min}}}{\sqrt{m}} \sum_{i = k}^{l(k)} \Xi_i \leq \frac{1}{2} \Xi_k + \sqrt{1-p_{\min}} \Xi_{k-1} + \hat c \Delta_{k, k+m},
\end{equation}
where $\hat c $ denotes the constant from Lemma~\ref{Lem:alpha}.
\end{lemma}

\begin{proof}
As $x \mapsto \sqrt{x}$ is a concave function, the application of Jensen's inequality yields
\begin{equation}\label{eq:inequality}
   \frac{1-\sqrt{1-p_{\min}}}{\sqrt{m}} \sum_{i = k}^{l(k)} \Xi_i \leq \big( 1-\sqrt{1-\pmin} \big) \sqrt{\Rvalue_k - \Rvalue_{k+m}}.
\end{equation}
We now distinguish two cases. The case $k \in K_1$ follows directly by calculation as in reference \cite{Qian2024}. In the second case $k \in K_2$, the key difference to the proximal case in \cite{kanzow2025} is that from Lemma~\ref{lemma:bs} together with the KL property, 
one obtains the bound
\begin{equation*}
	\chi'\big(\varphi(\xk) - \varphi(x^*)\big) \geq \frac{1}{\frac{2 b}{\taug_C} \| \xkp - \xk \|}.
\end{equation*}
By properties of the desingularization function $\chi$, one arrives at the estimate
\begin{equation*}
	\varphi(\xk) - \Rvalue_{k+m} \leq \frac{2 \hat c}{\pmin} \Xi_k \Delta_{k, k+m},
\end{equation*}
from which the claim can be deduced. The complete proof is in the appendix.
\end{proof}

Now, we obtain the global convergence of $\{\xk\}$ to a stationary point under the KL property of $\varphi$.

\begin{theorem}\label{thm:convergence}
Let Assumptions~\ref{descentas} and \ref{genas} hold, let $\{\xk\}_K$ be a subsequence converging to some accumulation point $x^*$, and suppose that $\Phi:= \varphi + \delta_\mathcal{M}$ satisfies the KL property at $x^*$. Then the entire sequence $\{\xk\}$ converges to $x^*$.
\end{theorem}

\begin{proof}
Again, we refer to the appendix for the complete proof. Let $k_0$ be the index from the definition of $\vartheta$, cf.\ Lemma~\ref{Lem:alpha}.
Without loss of generality, we may assume that $k_0$ is large enough, such that the results from Lemma~\ref{lemma:bs} and Lemma~\ref{lemma:lemmaxi} hold and that $\Rvalue_{k_0} < \varphi(x^*) + \eta$. Then, the following two statements hold:
\begin{enumerate}[label=(\alph*)]
	\item for all $k \geq k_0-1$: $\xk \in \clball_\vartheta(x^*)$, and
	\item for all $k \geq l(k_0)$:
	\begin{equation}\label{eq:statementb}
	\big( 1-\sqrt{1-\pmin} \big) 
		\sqrt{m} \sum_{j = l(k_0)}^k \Xi_j \leq \sum_{j=k_0}^k \left(\frac{1}{2}\Xi_j + \sqrt{1-\pmin} \Xi_{j-1}\right) + \hat c \sum_{j = k_0}^{l(k_0)} \chi \big( \Rvalue_j - \varphi(x^*) \big),
	\end{equation}
\end{enumerate}
where $ \hat{c} $ denotes the constant from Lemma~\ref{Lem:alpha}.

Both statements can be verified jointly by induction over $k$. Note that, in particular, due to the first statement, Lemma~\ref{lemma:lemmaxi} can be used in the induction step for the second claim. The calculations can be done similar to those in the references mentioned above. In the end, one obtains 
	\begin{equation*}
		\|\xkp - x^* \| \leq \|x^{k_0 - 1} - x^* \| + \sum_{j = k_0 - 1}^k \| \xjp - \xj \| \leq \vartheta.
	\end{equation*}
	Hence, taking $k \to \infty$, shows that $\{\xk\}$ is a Cauchy sequence and thus convergent.
\end{proof}

Lastly, let us cite the following rate-of-convergence result from \cite{Qian2024} for the case where the desingularization function is given by $\chi(t) = c t^\theta$ for some $c > 0$ and $\theta \in (0, 1)$. For completeness, we include a proof in the appendix.

\begin{theorem}\label{thm:convergencerate}
Let Assumptions~\ref{descentas} and \ref{genas} hold, and suppose that $\{ \xk \}$ converges on some subsequence $\{ \xk\}_K$ to a limit point $x^*$ such that $\varphi$ has the KL property at $x^*$. Then the entire sequence $\{ \xk \}$ converges to $x^*$. Further, if the corresponding desingularization function is given by $\chi(t) = ct^{\theta}$ (for some $c>0$ and $\theta \in (0,1)$), then the following statements hold:
\begin{enumerate}[label=(\alph*)]
	\item If $\theta \in [1/2, 1)$, then $\{\Rvalue_k\}$ converges R-linearly to $\varphi(x^*)$ and $\{\xk\}$ converges R-linearly to $x^*$. \label{itm1}
	\item If $\theta \in (0, 1/2)$, then there exist constants $\eta_1, \eta_2 > 0$ such that for all $k$ large enough it holds that\label{itm2}
	\begin{align}
		\Rvalue_k - \varphi(x^*) &\leq \eta_1 k^{-\frac{1}{1-2\theta}},\\
		\|\xk - x^*\| &\leq \eta_2 k^{-\frac{\theta}{1-2\theta}}.
	\end{align}
\end{enumerate}
\end{theorem}

\section{Numerical Experiments and Applications}\label{Sec:Numerics}

The nonmonotone subgradient method comes in different flavors. The simplest choice of the descent direction is the negative of the subgradient, i.e.\ $\dk = -\wk$. Then, as noted before, if suitable positive definiteness conditions on the corresponding matrices are ensured, the limited memory BFGS search direction from \cite{liu1989} also satisfies our condition in Assumption~\ref{genas}. Lastly, problem specific descent direction can be designed (as for the Euclidean minimum sum-of-squares clustering problem in \cite{aragon2025}).

Both, the choice of the stepsize strategy and the initial stepsize also have a significant influence on the numerical performance. For comparison, we will include the monotone version. We test the nonmonotone linesearch with parameters chosen manually ($p_k=0.6$ and $m_k=5$ for all $k \in \N$ for the mean- and max-rule respectively) in combination with the following initial stepsizes: For quasi-Newton methods, a constant initial guess, i.e.\ $\tauk = 1$ works well. If we take the negative of the subgradient, we also test a stepsize inspired by the Barzilai-Borwein stepsize \cite{Barzilai1988}, that is:
\begin{equation}\label{eq:bbstepsize}
	{\bar \tau}_k := \frac{\langle \Delta \xk, \Delta \xk\rangle}{\langle \Delta \xk, \Delta g^k\rangle},
\end{equation}
with $\Delta \xk = \xk - \xkm$ and $\Delta g^k = \wk - \proj_{T_{\xk}\mathcal{M}} (\wkm)$ (note that in a more general setting, we would need to employ parallel transport of $\wkm$ to the tangent space $T_{\xk} \mathcal{M}$). These stepsizes ${\bar \tau}_k$ are then projected onto the interval $[\taumin, \taumax]$. There are also so-called self-adaptive strategies, which were proposed in \cite{aragon2025}, where the nonmonotonicity parameter for the nonmonotone linesearch (i.e.\ the value $p_k$ for our method or the number of function values over which we take the maximum for the max-rule) and the initial stepsize is chosen automatically. Essentially, if for the two last iterations the initial stepsize was accepted, it is increased and we aim for a monotone decrease of the objective in the next iteration. Conversely, if the initial stepsize is not accepted, we may decrease the initial stepsize and allow for more nonmonotone behavior (for details, we refer to \cite{aragon2025}).

Defining
\begin{equation*}
	\Delta^k_{\text{rel}} x := \frac{\|\xk - \xkm\|}{\max\{\|\xkm\|, 1\}} \text{ \quad and \quad} \Delta^k_{\text{rel}} \varphi:= \frac{|\varphi(\xk) - \varphi(\xkm)|}{\max\{|\varphi(\xkm)|, 1\}},
\end{equation*}
for Euclidean problems, the stopping criterion is $\max \{ \Delta^k_{\text{rel}} x, \Delta^k_{\text{rel}} \varphi \} \leq \varepsilon$ and without linear structure, we use $\Delta^k_{\text{rel}} \varphi \leq \varepsilon$ with a tolerance parameter $\varepsilon > 0$.

\subsection{Euclidean Setting}

In the following, we present two applications of our method that arise in data science and machine learning: First, we test our method on the minimum sum-of-squares clustering problem. As a second application, we also show its applicability to multidimensional scaling problems. Both problems also admit a decomposition as a difference of convex functions. Therefore, it seems reasonable to compare the performance of our method to suitable algorithms from DC-programming. In particular, we consider the classical difference of convex functions algorithm (DCA) from \cite{LeThi1996} and the boosted difference of convex function algorithm (BDCA) introduced in \cite{aragon2020}, which features an additional linesearch that helps improve convergence speed. In both methods, we solved the subproblems using the LBFGS algorithm.

Let us note that both test problems can be applied to similar datasets, where the data features a partition into a set of clusters. In the case of the minimum sum-of-squares clustering this is clear and the goal is to identify the centers of these clusters. On the other hand, the multidimensional scaling technique is usually employed to inspect the data and decide whether the data can be grouped into meaningful clusters. We test the methods on both artificial and real data:
\begin{enumerate}
	\item \emph{Artificial Data:} The minimum sum-of-squares clustering is tested on data which is generated as follows: First, $l$ cluster centers are randomly generated from a normal distribution with a relatively high standard deviation. Secondly, the individual data points are again drawn from a normal distribution now centered around the cluster centers we obtained in the first step. The multidimensional scaling problem is tested on random normal data.
	\item \emph{Real Data:} We include experiments on the following datasets from the University of California Irvine Machine Learning Repository \cite{ucimlr}: \url{https://archive.ics.uci.edu/dataset/53/iris}, \url{https://archive.ics.uci.edu/dataset/186/wine+quality}, \url{https://archive.ics.uci.edu/dataset/545/rice+cammeo+and+osmancik}.
\end{enumerate}

\subsubsection{Minimum Sum-of-Squares Clustering}

Let $Y:= \{y^1, \dots, y^n\} \subset \R^s$ be a set of data points. The task is to split $Y$ into $l$ clusters given by centroids $\{x^1, \dots, x^l\}$, i.e.\ we aim to solve the problem
\begin{equation}\label{eq:minsos}
	\min \varphi(X) := \frac{1}{n} \sum_{j = 1}^n \min_{t \in \{1, \dots, l\}} \|x^t - y^j\|^2.
\end{equation}

Note that this problem was already studied in great detail in \cite{aragon2025} and previously in \cite{aragon2025a}, where the authors derive formulas for an element of the (Clarke-) subdifferential and also proposed a search direction based on second-order information from a local approximation to the objective function $\varphi$. In \cite{aragon2025}, the nonmonotone subgradient method with the max-rule was shown to outperform several other optimization methods. In particular, they were also competitive in runtime compared to the k-means clustering algorithm. Therefore, the main objective of our numerical study here is to demonstrate that the mean-rule nonmonotone method performs equally well in this application.

We compute the elements $\wk \in \partial_C \varphi$ as in \cite{aragon2025} as
\begin{equation*}
	\wk = \frac{1}{n} \sum_{j=1}^n \wk_j,
\end{equation*}
where $\wk_j \in \R^{s \times l}$ is zero except for one entry $2(\xk_t - y_j)$ at the position $t \in \arg\min_{t = 1, \dots, l} \| x_t - y_j \|$. The descent directions can be taken as in \cite{aragon2025}.

Let us also note that from \cite{ordin2015}, $\varphi$ admits the following DC formulation as $\varphi(X) = g(X) - h(X)$, 
where 
\begin{equation*}
	g(X) := \frac{1}{n} \sum_{i=1}^n \sum_{j=1}^l \| x_j - y_i\|^2 + \frac{\rho}{2} \|X\|^2,
\end{equation*}
and
\begin{equation*}
	h(X) := \frac{1}{n} \sum_{i=1}^n \max_{j=1, \dots, l} \sum_{t=1, t\neq j}^l \| x_t - y_i\|^2 + \frac{\rho}{2} \|X\|^2.
\end{equation*}
In both cases, $\|X\|^2$ denotes the Frobenius norm of $X$ and for $\rho > 0$ both functions are strongly convex. Note also that $g$ is differentiable and the subgradients of $h$ can be computed as in \cite{ordin2015}. 

\begin{figure}
    \centering
		\resizebox{0.48\textwidth}{!}{
			\includegraphics{"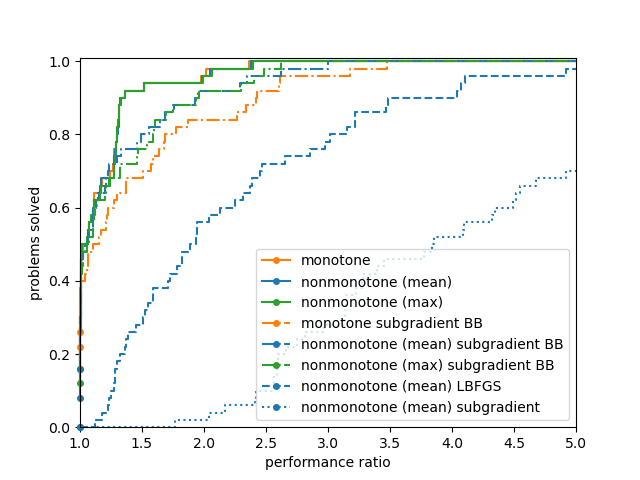"}
		}
	\caption{Performance profiles of CPU times for the Minimum Sum-of-Squares Clustering problem. The first $3$ methods use the descent direction computed in \cite{aragon2025}, $\varepsilon = 10^{-4}$.}\label{fig:msosperfprof}
	\label{fig:msosval}
\end{figure}

\begin{table}[ht]
\centering
\begin{tabular}{llrr}
\toprule
Dataset & Algorithm & Runtime (s) & Function Value\\
\midrule
\multirow{4}{*}{iris}
& nonmonotone (mean) & 0.116 & 0.420 \\
& nonmonotone (mean) LBFGS & 0.142 & 0.420 \\
& nonmonotone (mean) subgradient BB & 0.066 & 0.515 \\
& nonmonotone (mean) subgradient & 0.215 & 0.420 \\\hline
\multirow{4}{*}{wines}
& nonmonotone (mean) & 4.316 & 1322.897 \\
& nonmonotone (mean) LBFGS & 2.483 & 1322.894 \\
& nonmonotone (mean) subgradient BB & 1.955 & 1322.892 \\
& nonmonotone (mean) subgradient & 3.165 & 1322.897 \\\hline
\multirow{4}{*}{rice}
& nonmonotone (mean) & 1.208 & 1864620.633 \\
& nonmonotone (mean) LBFGS & 0.790 & 1864620.169 \\
& nonmonotone (mean) subgradient BB & 0.826 & 1864620.528 \\
& nonmonotone (mean) subgradient & 1.023 & 1864620.668 \\
\bottomrule
\end{tabular}
\caption{Numerical results for the minimum sum-of-squares clustering on real datasets (averaged over $N=5$ runs with random initialization, $\varepsilon = 10^{-4}$).}\label{tab:msosrealdata}
\end{table}

In Figure~\ref{fig:msosval}, we show performance profiles for the different methods applied to the minimum sum-of-squares clustering. Let us first note that both DCA as well as BDCA were also tested but are not included in the results as their results were significantly worse.

The methods with the descent direction from \cite{aragon2025} performed better than those using only first order information from either the subgradient directly or the LBFGS quasi-Newton direction. We note that the method using the inverse subgradient together with the stepsize from \eqref{eq:bbstepsize} performs almost as good as the choice of descent direction from \cite{aragon2025}. Further, in that case, both nonmonotone linesearches provide an additional speedup. In particular, both nonmonotone linesearches yield comparable results. This could also be observed for the other methods, i.e.\ the combination with LBFGS and the inverse subgradient with constant initial trial stepsize of $1$. Both, however, have longer runtimes in our experiment. Table~\ref{tab:msosrealdata} shows additional experiments using real datasets. Note that the initialization plays an important role for clustering algorithms. In our experiments, we choose the initial data centroids as random points from the datasets.

\subsubsection{Multidimensional Scaling}

Denote by $D := (\delta_{ij})_{i,j}$ a given dissimilarity matrix such that the entry $\delta_{ij}$ at the position $(i,j)$ measures the distance between two given data points $y_i \in \R^d$ and $y_j\in \R^d$, with $d$ large ($i, j \in \{1, \dots, n\}$). We aim to find points $\{x_1, \dots, x_n\}$ in a lower dimensional space (often in $\R^2$ or $\R^3$) such that the dissimilarities between the original data points $\{y_1, \dots, y_n\}$ correspond to the respective Euclidean distances in the lower dimensional space. Denoting $d_{ij}(X):= \| x_i -x_j\|$, the multidimensional scaling problem is given by 
\begin{equation}\label{eq:mds}
	\min_{X \in \R^{d \times n}} \varphi(X) := \sum_{i<j} (d_{ij}(X) - \delta_{ij})^2, 
\end{equation}
where $w_{ij}$ are some weight parameters.

Again, the formulation of the multidimensional scaling problem in \eqref{eq:mds} can be written as a DC program $\varphi(X) = g(X) - h(X)$,
where 
\begin{equation*}
	g(X) := \frac{1}{2} \sum_{i < j} w_{ij} d_{ij}^2(X) + \frac{\rho}{2} \|X\|^2,
\end{equation*}
and
\begin{equation*}
	h(X) := \sum_{i < j} w_{ij} d_{ij}(X) + \frac{\rho}{2} \|X\|^2.
\end{equation*}
Note that $g$ is differentiable and the subgradients of $h$ can be directly computed.

\begin{figure}
    \centering
		\resizebox{0.48\textwidth}{!}{
			\includegraphics{"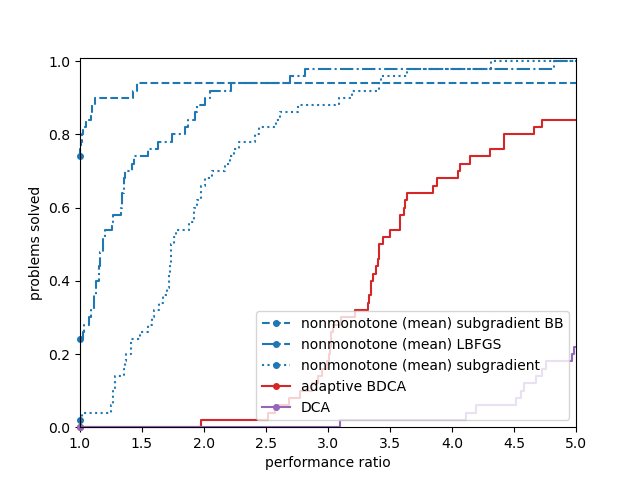"}
		}
	\caption{Performance profiles of CPU times for the Multidimensional Scaling problem, $\varepsilon = 10^{-4}$.}\label{fig:mdsperfprof}
	\label{fig:mdsfval}
\end{figure}

\begin{table}[ht]
\centering
\begin{tabular}{lrr}
\toprule
Method & Runtime (s)& Function Value \\
\midrule
Nonmonotone (mean) subgradient BB & 21.495 & -50942.952 \\
Nonmonotone (mean) LBFGS & 30.386 & -50961.759 \\
Nonmonotone (mean) subgradient & 52.662 & -50969.116 \\
Adaptive BDCA & 112.953 & -50963.378 \\
DCA & 312.823 & -50968.881 \\
\bottomrule
\end{tabular}
\caption{Numerical results for the multidimensional scaling problem on the iris dataset (averaged over $N=5$ runs with random initialization, $\varepsilon = 10^{-4}$).}\label{tab:mdsrealdata}
\end{table}

In Figure~\ref{fig:mdsfval}, the performance profiles show that in this case, the method with the stepsize from \eqref{eq:bbstepsize} together with our nonmonotone linesearch method seems to perform best, compared to the method in combination with the LBFGS descent direction and the inverse subgradient method with constant initial trial stepsize. Further, both DCA and BDCA have significantly longer runtimes. In Table~\ref{tab:mdsrealdata}, all algorithms are tested on the iris dataset.

\subsection{Clustering on Manifolds}

We will now propose an extension of the aforementioned minimum-sum-of-squares clustering to data on manifolds. Again, let $Y:= \{y^1, \dots, y^n\} \subset \mathcal{M}$ be a given set of data points on some manifold $\mathcal{M}$. The task is to split $Y$ into $l$ clusters given by centroids $\{x^1, \dots, x^l\} \subset \mathcal{M}$ such that the following optimization problem is solved:
\begin{equation}\label{eq:minsod}
	\min \varphi(X) := \frac{1}{n} \sum_{j = 1}^n \min_{t \in \{1, \dots, l\}} d(x^t, y^j).
\end{equation}
Here, $d(x, y)$ measures the dissimilarity between two points $x, y \in \mathcal{M}$ (but $d$ does not need to be a distance metric). In order to fit into our framework, we assume $d$ to be smooth (i.e.\ $\mathcal{C}^2$). In the Euclidean setting, $d(x, y) := \|x-y\|^2$ is the case considered in \eqref{eq:minsos}, and then this smoothness assumption holds true. 

Many of the clustering problems on manifolds have been separately analysed in a more statistical context. Our method is different as we take a geometric approach and cluster the data only based on their similarity and make no assumptions on e.g.\ the distribution.

We initialize the method by randomly sampling the initial estimates for the centroids as data points from the dataset. As the method might get stuck in local minima, at least for datasets of reasonable size, a good clustering can be obtained by comparing the results from several runs of the algorithm with different initializations.

Apart from the value of the objective function \eqref{eq:minsod}, we include the following evaluation metrics: homogeneity (a cluster is homogeneous if its members belong to a single class of the ground truth), completeness (members of a given class in the ground truth are assigned to the same cluster) and the V-measure (the harmonic mean of homogeneity and completeness) each take values between $0$ and $1$ \cite{Rosenberg2007}. The adjusted Rand index (measures cluster similarities through pairwise comparisons) takes values between $-0.5$ and $1$ \cite{Rand1971, Hubert1985}. Higher values indicate a better clustering.

Elements of the projectional subdifferentials are hard to compute. Thus, we take $\wk_\mathcal{M}$ from the upper bound $\wk_\mathcal{M} \in \proj_{T_{\xk} \mathcal{M}} \varphi(\xk)$ and assume that the sum-rule of the limiting subdifferential holds for the sum in problem \eqref{eq:minsod}. We take the initial stepsizes as in \eqref{eq:bbstepsize} and the termination criterion as before.

Finally, let us note that implementations of retractions and projections onto tangent spaces for many manifolds are available e.g.\ python in the package Pymanopt \cite{Townsend2016}.

\subsubsection{Spherical Clustering}

In the case of spherical data, a common measure for the similarity between two vectors $x$ and $y$ is given by the so-called \emph{cosine similarity} measure. This is simply the cosine of their angles, which corresponds to the inner product $\langle x, y\rangle$ as $x$ and $y$ have unit length. Our dissimilarity measure is hence defined as
\begin{equation}\label{eq:cosdissim}
	d_{\cos}(x, y) = 1- \langle x, y\rangle.
\end{equation}
Data on the hypersphere arises frequently in data science if the data vectors are scaled to have unit length, which is the case if e.g.\ frequencies are counted or one is dealing with directional data. The formulation in \eqref{eq:cosdissim}, together with \eqref{eq:minsod}, was already used in \cite{Dhillon2001} to cluster unlabeled document data based on word counts. On the other hand, directional data also arises naturally in many natural and physical sciences \cite{Golzy2016}.

Let us note that the cosine dissimilarity is directly related to the Euclidean distance of two vectors of unit length. In that case, we have $\|x-y\|^2 = 2 d_{\cos}(x, y)$.

On the sphere $\mathbb{S}^{n-1}$, by \cite{Absil2008}, we have the retraction $\retr_x (\xi) := \frac{x+\xi}{\|x+\xi\|}$,
and the projection onto tangent spaces are given by $\proj_{T_x\mathbb{S}^{n-1}} \xi = (I- xx^T) \xi$ for all $x \in \mathbb{S}^{n-1}$ and $\xi \in T_x \mathbb{S}^{n-1}$.

We show that our method can be applied to spherical clustering by comparing its results to those obtained from the sperical k-means clustering algorithm in \cite{Dhillon2001}. We consider two test scenarios:
\begin{enumerate}
	\item \textit{Artificial data: } We generate $k$ uniformly distributed random data points on the sphere as mean directions of von Mises-Fisher distributed data clusters with some concentration parameter $\kappa$. Each cluster is of the same size $l$.
	\item \textit{Real data: } We use the 20 newsgroups text dataset (available online e.g.\ here: \url{https://archive.ics.uci.edu/dataset/113/twenty+newsgroups}) of which we randomly select 5 categories in each run. The text data is then preprocessed by removing English stop words, vectorizing the data, counting word occurrences in each document and converting them to their so-called tf-idf measure. Subsequently, we use a singular value decomposition to reduce the dimension to $d=100$ and normalize each resulting vector. Hence, all data vectors lie on a high dimensional unit sphere.
\end{enumerate}

\begin{figure}
    \centering
		\resizebox{0.48\textwidth}{!}{
			\includegraphics{"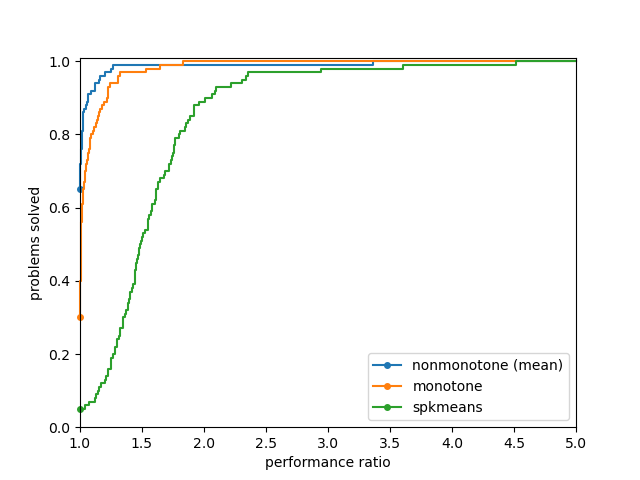"}
		}
	\caption{Performance profiles of CPU times for clustering on the sphere with artificial data ($100$ samples).}
	\label{fig:sphereruntimes}
\end{figure}

\begin{table}[ht]
\centering
\begin{tabular}{lccc}
\toprule
 & Monotone & Nonmonotone & Spherical k-means\\
\midrule
Runtime (s) & $0.3546 \pm 0.1069$ & $0.3480 \pm 0.1279$ & $0.5188 \pm 0.1824$ \\
Function Value & $0.840 \pm 0.002$ & $0.840 \pm 0.002$ & $0.840 \pm 0.002$ \\
Homogeneity & $0.499 \pm 0.027$ & $0.501 \pm 0.022$ & $0.498 \pm 0.027$ \\
Completeness & $0.499 \pm 0.027$ & $0.501 \pm 0.022$ & $0.498 \pm 0.027$ \\
V-measure & $0.499 \pm 0.027$ & $0.501 \pm 0.022$ & $0.498 \pm 0.027$ \\
Adjusted Rand-Index & $0.539 \pm 0.035$ & $0.542 \pm 0.028$ & $0.538 \pm 0.035$ \\
\bottomrule
\end{tabular}
	\caption{Results for spherical clustering on artificial data (averaged over $100$ samples, with standard deviations, $\varepsilon = 10^{-6}$).}
\label{tab:vmfclustering}
\end{table}

For an artificial dataset with $5$ clusters, $1000$ data points per cluster on the $\mathbb{S}^{200}$ and with a concentration parameter $\kappa = 10$, our numerical results are summarized in Table~\ref{tab:vmfclustering} with a performance profile for the runtime depicted in Figure~\ref{fig:sphereruntimes}. In this case, the solutions found by our method and the spherical k-means algorithm are of the same quality. However, our method is faster with the nonmonotone version providing an additional but only very slight advantage. In our experiments, we found that for large values of $\kappa$, i.e.\ more concentrated data, the spherical k-means algorithm performs better, whereas for small concentration parameters, our method is superior.

\begin{figure}
    \centering
		\resizebox{0.48\textwidth}{!}{
			\includegraphics{"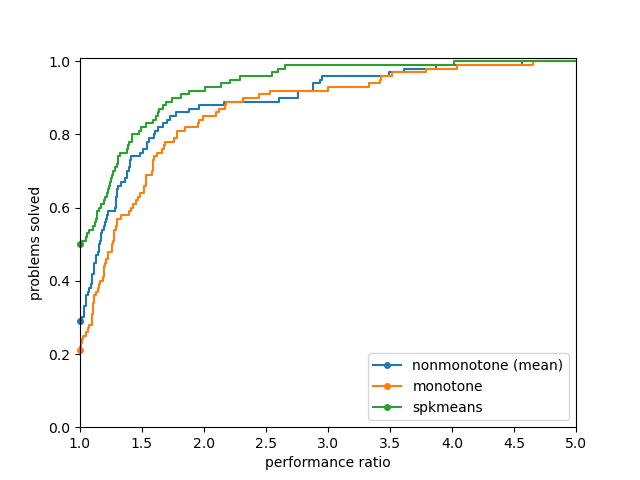"}
		}
	\caption{Performance profiles of CPU times for document clustering on the sphere ($100$ samples).}
	\label{fig:docruntimes}
\end{figure}

\begin{table}[ht]
\centering
\begin{tabular}{lccc}
\toprule
 & Monotone & Nonmonotone & Spherical k-means\\
\midrule
Runtime (s) & $0.266 \pm 0.135$ & $0.243 \pm 0.110$ & $0.224 \pm 0.107$ \\
Function Value & $0.628 \pm 0.010$ & $0.628 \pm 0.010$ & $0.626 \pm 0.009$ \\
Homogeneity & $0.491 \pm 0.073$ & $0.488 \pm 0.076$ & $0.506 \pm 0.072$ \\
Completeness & $0.509 \pm 0.070$ & $0.507 \pm 0.074$ & $0.522 \pm 0.071$ \\
V-measure & $0.500 \pm 0.071$ & $0.497 \pm 0.074$ & $0.514 \pm 0.071$ \\
Adjusted Rand-Index & $0.471 \pm 0.089$ & $0.467 \pm 0.093$ & $0.490 \pm 0.087$ \\
\bottomrule
\end{tabular}
	\caption{Results for spherical clustering on document clustering data (averaged over $100$ samples, with standard deviations, $\varepsilon = 10^{-6}$).}
\label{tab:docclustering}
\end{table}

Results for the real dataset for document clustering are reported in Figure~\ref{fig:docruntimes} and Table~\ref{tab:docclustering}. Our method has a slightly worse performance compared to the spherical k-means algorithm. 

\subsubsection{Clustering on Matrix Manifolds}

The approach for directional data, i.e.\ data on the hypersphere $\mathbb{S}^{n-1}$ can be directly extended to the so-called Stiefel manifold
\begin{equation*}
	\mathrm{St}(p, n) := \{ X \in \R^{n \times p} \mid X^T X = I_p\}, 
\end{equation*}
which inherits the scalar product $\langle X, Y \rangle = \tr( X^T Y)$ on its tangent space from $\R^{n \times p}$. Projections onto the tangent space are given by (see \cite{Absil2008}):
\begin{equation*}
	\proj_{T_X \mathrm{St}(p, n)} (\xi) := (I-XX^T) \xi + \frac{1}{2} X (X^T \xi - \xi^T X).
\end{equation*}
There exist different retractions, for example one might chose 
\begin{equation*}
	\retr_X(\xi) := (X+ \xi)(I_p - \xi^T \xi)^{-1/2},
\end{equation*}
or alternatively a QR decomposition of $X+ \xi$ can be used (see \cite{Absil2008} for details). Clustering problems on the Stiefel manifold have been considered, e.g.\ in \cite{Cetingul2009, Sengupta2017}. As suggested in \cite{Chikuse2003, Cetingul2009}, we take the following dissimilarity measure
\begin{equation*}
	d(X, Y) := p - \mathrm{tr}(X^T Y)
\end{equation*}
in \eqref{eq:minsod}.

The Grassmann manifold $\mathrm{Gr}(p, n)$ is the set of all $p$ dimensional subspaces of $\R^n$. We refer to \cite{Bendokat2024} for an overview of the different representations and only mention two approaches: First, we can identify every subspace $\mathcal{U} \subseteq \R^n$ with its associated orthogonal projection matrix $P \in \R^{n \times n}$. With this representation, the Grassmannian can be written as the set of all $P \in \R^{n \times n}$ such that $P^T = P$, $P^2 = P$ and $\mathrm{rank}P = p$. On the other hand, the Grassmannian can also be described as the quotient manifold $\mathrm{St}(p, n)/\mathrm{O}(p)$ (where $\mathrm{O}(p)$ is the orthogonal group). These equivalent formulations are linked by the fact that the orthogonal projector onto a subset $\mathcal{U} \subseteq \R^n$ spanned by the columns of $U \in \mathrm{St}(p, n)$ is given by $P=UU^T$.

Clustering on the Grassmann manifold was considered in \cite{Gruber2006, Cetingul2009} and \cite{Oh2019}, the latter provides an application to the clustering of multivariate times series. Following \cite{Chikuse2003, Cetingul2009, Oh2019}, we can measure the dissimilarity of two points $P_1 = UU^T$ and $P_2 = VV^T$, where $U, V \in \mathrm{St}(p, n)$, on the Grassmannian with
\begin{equation*}
	d(P_1, P_2) = p - \mathrm{tr}(UU^TVV^T) = p - \mathrm{tr}(P_1 P_2).
\end{equation*}

Random matrices on these two matrix manifolds are generated with the method described in \cite{Chikuse2003} and \cite{Cetingul2009}: Let $O \in \R^{n \times n}$ be an orthogonal matrix. Then $O$ can be represented as a product of (orthogonal) matrices $O^k_n(\theta)$ that take the form
\begin{equation*}
	O^k_n(\theta) := \begin{pmatrix} I_{k-1} & 0 &0 &0\\ 0  & \cos(\theta) & - \sin(\theta) & 0\\ 0 & \sin(\theta) & \cos(\theta) & 0\\ 0 & 0 &0 & I_{n-k-1}\end{pmatrix},
\end{equation*}
as $O = \prod_{k = 1}^{n-1} \prod_{j = k}^{n-1} O_n^j(\theta_{k, j})$, where $\theta_{k,j}$ are a set of angles. By selecting the first $p$ columns, we obtain an element $X \in \mathrm{St}(p,n)$ and $P := XX^T \in \mathrm{Gr}(p,n)$. For each cluster, such a set of angles is chosen and then (for the individual data points) we add a uniform random noise on the interval $[-\pi/9, \pi/9]$. 

\begin{table}[ht]
\centering
\begin{tabular}{lcc}
\toprule
 & Monotone & Nonmonotone \\
\midrule
Runtime (s) & $0.66 \pm 0.44$ & $0.54 \pm 0.32$ \\
Function Value & $0.847 \pm 0.288$ & $0.837 \pm 0.290$ \\
Homogeneity & $0.882 \pm 0.108$ & $0.885 \pm 0.109$ \\
Completeness & $0.940 \pm 0.054$ & $0.942 \pm 0.055$ \\
V-measure & $0.909 \pm 0.084$ & $0.911 \pm 0.084$ \\
Adjusted Rand-Index & $0.814 \pm 0.164$ & $0.820 \pm 0.166$ \\
\bottomrule
\end{tabular}
	\caption{Clustering results for clustering data on the Stiefel manifold (averaged over $100$ examples, with standard deviations, $\varepsilon = 10^{-6}$).}
\label{tab:stiefelclustering}
\end{table}

\begin{table}[ht]
\centering
\begin{tabular}{lcc}
\toprule
 & Monotone & Nonmonotone \\
\midrule
Runtime (s) & $0.85 \pm 0.46$ & $0.81 \pm 0.40$ \\
Function Value & $0.542 \pm 0.091$ & $0.538 \pm 0.087$ \\
Homogeneity & $0.862 \pm 0.133$ & $0.869 \pm 0.130$ \\
Completeness & $0.911 \pm 0.081$ & $0.917 \pm 0.080$ \\
V-measure & $0.885 \pm 0.109$ & $0.891 \pm 0.107$ \\
Adjusted Rand-Index & $0.807 \pm 0.189$ & $0.817 \pm 0.186$ \\
\bottomrule
\end{tabular}
	\caption{Clustering results for clustering on the Grassmann manifold (averaged over $100$ examples, with standard deviations, $\varepsilon = 10^{-6}$).}
\label{tab:grassmannclustering}
\end{table}

Results are shown Table~\ref{tab:stiefelclustering} for data on $\mathrm{St}(5, 10)$, with $5$ clusters, each of $100$ data points and Table~\ref{tab:grassmannclustering} for clustering on the Grassmann manifold $\mathrm{Gr}(5, 10)$. In both cases, the V-measure and adjusted Rand-index are close to $1$, indicating that the clustering obtained by our method successfully reconstructs the structure of the artificial dataset. Further, we observe that the nonmonotone linesearch performs at least as good as its monotone counterpart.

\section{Final Remarks}\label{Sec:Final}

The motivation behind our work was mainly twofold: First, we noticed that by employing the mean-rule as nonmonotone linesearch procedures in our algorithm, the analysis turns out to be much simpler compared to the related max-rule used in \cite{aragon2025}. By adapting the corresponding convergence theory for nonmonotone descent methods, we were further able to provide stronger theoretical guarantees in presence of the Kurdyka–Łojasiewicz property. Secondly, in the presence of specific constraints, namely if we consider minimization problems over submanifolds, these findings were further generalized through techniques from manifolds optimization. 

It remains an open question whether other related methods, e.g.\ a projected (nonmonotone) subgradient method, can be realized under the same assumptions on the objective function (i.e.\ upper-$\mathcal{C}^2$) and with similar convergence properties.

\printbibliography

\appendix

\section{Complete Proofs under the KL Property}

Let us begin with a full proof of Lemma~\ref{lemma:lemmaxi}:

\begin{lemma}\label{lemma:lemmaxiproof}
For all sufficiently large $k \in \mathbb{N}$ with $\Rvalue_k < \varphi(x^*) + \eta$ and $\xk \in \clball_\vartheta(x^*)$, the following inequality holds:
\begin{equation}\label{eq:lemmaxi}
	\frac{1-\sqrt{1-p_{\min}}}{\sqrt{m}} \sum_{i = k}^{l(k)} \Xi_i \leq \frac{1}{2} \Xi_k + \sqrt{1-p_{\min}} \Xi_{k-1} + \hat c \Delta_{k, k+m},
\end{equation}
where $\hat c $ denotes the constant from Lemma~\ref{Lem:alpha}.
\end{lemma}

\begin{proof}
As $x \mapsto \sqrt{x}$ is a concave function, the application of Jensen's inequality yields
\begin{equation}\label{eq:inequality}
   \frac{1-\sqrt{1-p_{\min}}}{\sqrt{m}} \sum_{i = k}^{l(k)} \Xi_i \leq \big( 1-\sqrt{1-\pmin} \big) \sqrt{\Rvalue_k - \Rvalue_{k+m}}.
\end{equation}
We now distinguish two cases.\smallskip

\noindent
\textbf{Case 1:} $k \in K_1$. 
Then, we have $ \varphi (\xk) \leq \Rvalue_{k+m} $, which implies
\begin{align*}
   \Rvalue_k - \Rvalue_{k+m} &= (1-p_k) \Rvalue_{k-1} + p_k \varphi(\xk) - \Rvalue_{k+m}\\
   & \leq (1-p_k) \Rvalue_{k-1} + p_k \Rvalue_{k+m} - \Rvalue_{k+m}\\
   & = (1-p_k) (\Rvalue_{k-1} - \Rvalue_{k+m})\\
	&\leq (1-\pmin) (\Rvalue_{k-1} - \Rvalue_{k+m})
	\qquad (\text{recall that } \Rvalue_{k-1} \geq \Rvalue_{k+m})\\
	&= (1-\pmin) (\Rvalue_{k-1} - \Rvalue_k + \Rvalue_k - \Rvalue_{k+m}).
\end{align*}
Using $\sqrt{x+y} \leq \sqrt{x} + \sqrt{y}$ for all $x, y \in \mathbb{R}_{\geq 0}$, we obtain
\begin{equation*}
   \big( 1- \sqrt{1- \pmin} \big) \sqrt{\Rvalue_k - \Rvalue_{k+m}} \leq \sqrt{1-\pmin} \Xi_{k-1}.
\end{equation*}
The statement therefore follows from \eqref{eq:inequality}. \smallskip

\noindent
\textbf{Case 2:} $k \in K_2$. We then have $\varphi(x^*) \leq 
\Rvalue_{k+m} < \varphi(\xk) \leq \Rvalue_{k} < \varphi(x^*) + \eta$
by assumption. Using the KL property of $\Phi$ and the fact that $\Phi$ agrees with $\varphi$ on $\mathcal{M}$, we get
\begin{equation*}
   \chi' \big( \varphi(\xk) - \varphi(x^*) \big) \dist \big( 0, \partial \Phi(\xk) \big) = 
   \chi' \big( \Phi(\xk) - \Phi(x^*) \big) \dist \big( 0, \partial \Phi(\xk) \big) \geq 1.
\end{equation*}
As $\xk$ was assumed to be in $\clball_\vartheta(x^*)$, by application of \eqref{distbound}, we have
\begin{equation}\label{eq:chi-star}
	\chi'(\varphi(\xk) - \varphi(x^*)) \geq \frac{1}{\frac{2 b}{\taug_C} \| \xkp - \xk \|}.
\end{equation}
(recall that $ \xkp \neq \xk $ since the algorithm
is assumed to generate an infinite sequence).
Using the properties of $\chi$, we now obtain
\begin{align*}
   \Delta_{k, k+m} &= \chi \big(\Rvalue_k- \varphi(x^*) \big) -
    \chi \big( \Rvalue_{k+m} - \varphi(x^*) \big) \\
	&\geq \chi \big( \varphi(\xk)- \varphi(x^*) \big) - 
	\chi \big( \Rvalue_{k+m} - \varphi(x^*) \big)\\
	&\geq \chi' \big( \varphi(\xk) - \varphi(x^*) \big) 
	\big( \varphi(\xk) - \Rvalue_{k+m} \big)\\
	&\geq \frac{\varphi(\xk) - \Rvalue_{k+m}}{\frac{2 b}{\taug_C} \| \xkp - \xk \|},
\end{align*}
where the first inequality results from the monotonicity of 
$ \chi $, the next one exploits the concavity of $ \chi $, and 
the final estimate exploits \eqref{eq:chi-star} together with the
fact that $ \varphi(\xk) - \Rvalue_{k+m} > 0 $ in the case under
consideration. Thus, with \eqref{eq:normxi}, we get
\begin{equation*}
   \varphi(\xk) - \Rvalue_{k+m} \leq \frac{2 \hat c}{\pmin} \Xi_k \Delta_{k, k+m},
\end{equation*}
from the definition of $ \hat{c} $.
Similar to the first case, our aim is to bound 
the difference $\Rvalue_k - \Rvalue_{k+m}$. Using the fact that $\varphi(\xk) \leq \Rvalue_{k-1}$ by the acceptance criterion for our 
stepsize computation as well as $p_k \geq \pmin$, we have
\begin{equation*}
	p_k \varphi(\xk) + ( 1- p_k)\Rvalue_{k-1} \leq 
    p_{\min} \varphi(\xk) + ( 1 - p_{\min}) \Rvalue_{k-1} .
\end{equation*}
Together with the definition of $ \Rvalue_k := (1-p_k)\Rvalue_{k-1} + p_k \varphi(\xk) $, this yields
\begin{align*}
   \Rvalue_k - \Rvalue_{k+m} &= p_k \varphi(\xk) + (1-p_k) \Rvalue_{k-1} - \Rvalue_{k+m} \\
	& \leq \pmin \varphi(\xk) + (1-\pmin) \Rvalue´_{k-1} - \pmin \Rvalue_{k+m} - (1-\pmin) \Rvalue_{k+m}\\
	& = \pmin \big( \varphi(\xk) - \Rvalue_{k+m} \big) + (1-\pmin) (\Rvalue_{k-1} - \Rvalue_{k+m})\\
	& \leq 2\hat c \Xi_k \Delta_{k, k+m}  + (1-\pmin) (\Rvalue_{k-1} - \Rvalue_{k+m})\\
	& = 2\hat c \Xi_k \Delta_{k, k+m}  + (1-\pmin) (\Rvalue_{k-1} - \Rvalue_k + \Rvalue_k - \Rvalue_{k+m}).
\end{align*}
Taking square roots on both sides and using
$\sqrt{x+y} \leq \sqrt{x} + \sqrt{y}$ for all $ x, y \in \mathbb{R}_{\geq 0}$, we obtain
\begin{equation*}
	\sqrt{\Rvalue_k - \Rvalue_{k+m}}
	\leq \sqrt{2 \hat{c} \Xi_k \Delta_{k,k+m}} +
	\sqrt{1 - \pmin} \big( \sqrt{\Rvalue_{k-1} - \Rvalue_k} 
	+ \sqrt{\Rvalue_k - \Rvalue_{k+m}} \big),
\end{equation*}
so we have
\begin{equation*}
   \big( 1 - \sqrt{1 - \pmin} \big) \sqrt{\Rvalue_k - \Rvalue_{k+m}}
   \leq \sqrt{2 \hat{c} \Xi_k \Delta_{k,k+m}} +
   \sqrt{1 - \pmin} \Xi_{k-1}.
\end{equation*}
Exploiting the inequality $2\sqrt{xy} \leq x+y$ for all $x, y \in \mathbb{R}_{\geq 0}$, this yields
\begin{equation*}
   \big( 1-\sqrt{1-\pmin} \big) 
   \sqrt{\Rvalue_k - \Rvalue_{k+m}} \leq 
	\frac{1}{2} \Xi_k + \sqrt{1-\pmin} \Xi_{k-1} + \hat c \Delta_{k, k+m}.
\end{equation*}
In view of \eqref{eq:inequality}, this completes the proof.
\end{proof}

\begin{theorem}\label{thm:convergenceproof}
Let Assumptions~\ref{descentas} and \ref{genas} hold, let $\{\xk\}_K$ be a subsequence converging to some accumulation point $x^*$, and suppose that $\Phi:= \varphi + \delta_\mathcal{M}$ satisfies the KL property at $x^*$. Then the entire sequence $\{\xk\}$ converges to $x^*$.
\end{theorem}

\begin{proof}
Let $k_0$ be the index from the definition of $\vartheta$, cf.\ Lemma~\ref{Lem:alpha}. Without loss of generality, we may assume that $k_0$ is large enough such that the results from Lemma~\ref{lemma:bs} and Lemma~\ref{lemma:lemmaxi} hold and that $\Rvalue_{k_0} < \varphi(x^*) + \eta$. We now claim that the following statements hold:
\begin{enumerate}[label=(\alph*)]
	\item for all $k \geq k_0-1$: $\xk \in \clball_\vartheta(x^*)$, and
	\item for all $k \geq l(k_0)$:
	\begin{equation}\label{eq:statementb}
	\big( 1-\sqrt{1-\pmin} \big) 
		\sqrt{m} \sum_{j = l(k_0)}^k \Xi_j \leq \sum_{j = k_0}^k \left(\frac{1}{2} \Xi_j + \sqrt{1-\pmin} \Xi_{j-1}\right) + \hat c \sum_{j = k_0}^{l(k_0)} \chi \big( \Rvalue_j - \varphi(x^*) \big),
	\end{equation}
\end{enumerate}
where $ \hat{c} $ denotes the constant from Lemma~\ref{Lem:alpha}.
We verify these two statements simultaneously by induction over $k$. 

For $k = k_0-1$, the term $\| x^{k_0-1} - x^* \|$ is obviously smaller than the constant $\vartheta$ in Lemma~\ref{Lem:alpha}. Hence $x^{k_0-1} \in \clball_\vartheta(x^*)$. Now, assuming $\xj \in \clball_\vartheta(x^*)$ for all $j$ from $k_0-1$ to some $k \in \{k_0-1, \dots, l(k_0)\}$, we can apply \eqref{eq:normxi} to obtain
\begin{align*}
	\|\xk - x^*\| &\leq \| x^{k_0-1} - x^* \| + \sum_{j = k_0}^{k} \| \xj - \xjm \|\\
	& \leq \| x^{k_0-1} - x^* \| + \frac{1}{e} \sum_{j = k_0}^{k} \Xi_{j-1}\\
	& \leq \| x^{k_0-1} - x^* \| + \frac{1}{e} \sum_{j = k_0}^{l(k_0)} \Xi_{j-1} \leq \vartheta.
\end{align*}
This shows the first statement for $k = k_0-1, \dots, l(k_0)$. Now, we can apply Lemma~\ref{lemma:lemmaxi} for indices $k = k_0, \dots, l(k_0)$ and obtain
\begin{align*}
   \big( 1-\sqrt{1-\pmin} \big)
   \sqrt{m} \Xi_{l(k_0)}& \leq \frac{1 - \sqrt{1- \pmin}}{\sqrt{m}} \sum_{j=k_0}^{l(k_0)} \sum_{i = j}^{l(j)} \Xi_i \\
	&\leq \sum_{j=k_0}^{l(k_0)} \left(\frac{1}{2} \Xi_j + \sqrt{1-p_{\min}} \Xi_{j-1}\right) + \hat c \sum_{j = k_0}^{l(k_0)} \Delta_{j, j+m}\\
	&\leq  \sum_{j=k_0}^{l(k_0)} \left(\frac{1}{2} \Xi_j + \sqrt{1-p_{\min}} \Xi_{j-1}\right) + \hat c \sum_{j = k_0}^{l(k_0)} 
	\chi \big( \Rvalue_j - \varphi(x^*) \big),
\end{align*}
where the first inequality follows from the fact that the term
$ \Xi_{l(k_0)} $ occurs $ m $ times within the double sum on the
right-hand side, whereas the other expressions $ \Xi_i $ are
nonnegative, and the second inequality is obtained by summing \eqref{eq:lemmaxi} in Lemma~\ref{lemma:lemmaxi} from $k_0$ to $l(k_0)$. 
In the final estimate, we simply omit some nonpositive terms.
This shows that the second statement holds for $k=l(k_0)$.

Suppose that the first statement holds for all $j$ from $k_0-1$ to some $k \geq l(k_0)$ and that the second statement is true for $k \geq l(k_0)$. We first show that the second statement for $k$ implies the first statement for $k+1$. Using \eqref{eq:statementb}, we get
\begin{align*}
   & \big( 1-\sqrt{1-\pmin} \big) 
   \sqrt{m}\sum_{j = l(k_0)}^k \Xi_j \\
	& \hspace*{8mm} \leq \sum_{j = k_0}^k \left(\frac{1}{2}\Xi_j+\sqrt{1-p_{\min}}   \Xi_{j-1}\right) + \hat c \sum_{j = k_0}^{l(k_0)} \chi \big( \Rvalue_j - \varphi(x^*) \big) \\
	& \hspace*{8mm} \leq \sum_{j = k_0}^{l(k_0)} \left(\frac{1}{2} \Xi_j + \sqrt{1-p_{\min}} \Xi_{j-1}\right) + \left( \frac{1}{2} + \sqrt{1- \pmin}\right) \sum_{j = l(k_0)}^{k} \Xi_j + \hat c \sum_{j = k_0}^{l(k_0)} \chi \big( \Rvalue_j - \varphi(x^*) \big).
\end{align*}
This implies 
\begin{align*}
   & \left( \Big(1-\sqrt{1-\pmin} \Big)\sqrt{m} - \Big(\frac{1}{2} + \sqrt{1-\pmin} \Big) \right) \sum_{j = l(k_0)}^k \Xi_j \\
	& \hspace*{8mm} \leq \sum_{j = k_0}^{l(k_0)} \left(\frac{1}{2} \Xi_j + \sqrt{1-p_{\min}} \Xi_{j-1}\right) + \hat c \sum_{j = k_0}^{l(k_0)}
   \chi \big( \Rvalue_j - \varphi(x^*) \big).
\end{align*}
Noting that, by definition of $m$, it holds that 
$\big( 1-\sqrt{1-\pmin} \big) \sqrt{m} - 
\big( 1/2 + \sqrt{1- \pmin} \big) \geq 1/2$ and that we obviously
have $\sqrt{1-\pmin} \leq 1$, we get
\begin{equation*}
	\sum_{j=l(k_0)}^k \Xi_j \leq \sum_{j=k_0}^{l(k_0)} (\Xi_j + 2 \Xi_{j-1}) + 2 \hat c \sum_{j = k_0}^{l(k_0)} \chi \big( \Rvalue_j - \varphi(x^*) \big).
\end{equation*}
This implies
\begin{equation}\label{eq:sumXis}
	\sum_{j=k_0-1}^k \Xi_j = \sum_{j=k_0}^{l(k_0)} \Xi_{j-1} + \sum_{j=l(k_0)}^k \Xi_j \leq \sum_{j=k_0}^{l(k_0)} (3 \Xi_{j-1} + \Xi_j) + 2 \hat c \sum_{j = k_0}^{l(k_0)}\chi \big( \Rvalue_j - \varphi(x^*) \big).
\end{equation}
As we assumed that the first statement and by consequence $\xj \in \clball_\vartheta(x^*) \cap \mathcal {M} \subset C$ is true for all $j \in \{k_0 -1, \dots, k\}$, we obtain from \eqref{eq:normxi} together with the equation above that
\begin{equation}\label{eq:normxkp}
\begin{aligned}
	\|\xkp - x^*\| & \leq \| x^{k_0-1} - x^* \| + \sum_{j = k_0-1}^k \| \xjp - \xj\| \leq \| x^{k_0-1} - x^* \| + \frac{1}{e} \sum_{j=k_0-1}^k \Xi_j\\
	&\leq \| x^{k_0-1} - x^*\| + \frac{1}{e} \sum_{j=k_0}^{l(k_0)} (3 \Xi_{j-1} + \Xi_j) + \frac{2 \hat c}{e} \sum_{j = k_0}^{l(k_0)}\chi \big( \Rvalue_j - \varphi(x^*) \big).
\end{aligned}
\end{equation}
The expression on the right-hand side is precisely the constant
$ \vartheta $ from Lemma~\ref{Lem:alpha}. Thus, the first statement holds 
for $k+1$.

We next verify the second claim for $k+1$. Since we already know that $\xj \in \clball_\vartheta(x^*)$ is true for all $j \in \{k_0-1, \dots, k+1\}$, we again apply Lemma~\ref{lemma:lemmaxi} and sum over \eqref{eq:lemmaxi}, now from $k_0$ to $k+1$. This yields
\begin{align*}
   \big( 1-\sqrt{1-\pmin} \big) \sqrt{m} 
   \sum_{j=l(k_0)}^{k+1} \Xi_j & 
   \leq \frac{1-\sqrt{1-\pmin}}{\sqrt{m}} \sum_{j=k_0}^{k+1} \sum_{i=j}^{l(j)} \Xi_i\\
	&\leq \sum_{j=k_0}^{k+1} \left(\frac{1}{2} \Xi_j + \sqrt{1-\pmin} \Xi_{j-1}\right) + \hat c \sum_{j =k_0}^{k+1} \Delta_{j, j+m}\\
	&\leq \sum_{j=k_0}^{k+1} \left(\frac{1}{2} \Xi_j + \sqrt{1-\pmin} \Xi_{j-1}\right) + \hat c \sum_{j =k_0}^{l(k_0)} \chi \big( \Rvalue_j - \varphi(x^*) \big),
\end{align*}
where the first inequality results from the fact that each term
$ \Xi_j $ for $ j = l(k_0), \ldots, k+1 $ from the left-hand side
occurs $ m $ times within the double sum from the right-hand side
(observe that the relation $ l(j+1) = l(j)+1 $ holds for all $ j $),
whereas the remaining expressions $ \Xi_i $ are nonnegative,
the next inequality exploits \eqref{eq:lemmaxi} from 
Lemma~\ref{lemma:lemmaxi}, and the final estimate uses a telescoping 
sum argument where we omit some nonpositive summands. This completes our induction step.

Hence, it follows that $\xk \in \clball_\vartheta(x^*)$ for all $k \geq k_0-1$.
Taking $k \to \infty$ in the resulting expression for $\sum_{j = k_0-1}^k
\| \xjp- \xj \|$ from \eqref{eq:normxkp} then shows that 
$\{\xk \}_{k \in \mathbb{N}}$ is a Cauchy sequence and, therefore,
convergent. Thus, the accumulation point $x^*$ is the limit of the 
entire sequence $ \{ \xk \} $.
\end{proof}

For completeness, we provide a proof for the following rate-of-convergence result in Theorem~\ref{thm:convergencerateproof} for the case where the desingularization function is given by $\chi(t) = c t^\theta$ for some $c > 0$ and $\theta \in (0, 1)$. Note, however, that it is also a direct consequence of the corresponding more general result found in \cite{Qian2024}. We need the following technical result from \cite[Lemma 1]{Aragon2018}:

\begin{lemma}\label{Lem:Aragon}
Let $ \{ s_k \} \subseteq [0, \infty )$ be any monotonically decreasing sequence satisfying $ s_k \to 0 $ and 
\begin{equation*}
	s_k^\alpha \leq \beta \big( s_k - s_{k+1} \big) \quad
	\text{for all $ k $ sufficiently large},
\end{equation*}
where $ \alpha, \beta > 0 $ are suitable constants. Then the following
statements hold:
\begin{itemize}
	\item[(a)] If $ \alpha \in (0,1] $, the sequence $ \{ s_k \} $
	converges linearly to zero with rate $ 1 - \frac{1}{\beta} $.
	\item[(b)] If $ \alpha > 1 $, there exists a constant $ \eta > 0 $
	such that 
	\begin{equation*}
		s_k \leq \eta k^{- \frac{1}{\alpha - 1}} \quad
		\text{for all $ k $ sufficiently large}.
	\end{equation*}
\end{itemize}
\end{lemma}

\begin{theorem}\label{thm:convergencerateproof}
Let Assumptions~\ref{descentas} and \ref{genas} hold, and suppose that $\{ \xk \}$ converges on some subsequence $\{ \xk\}_K$ to a limit point $x^*$ such that $\Phi$ has the KL property at $x^*$. Then the entire sequence $\{ \xk \}$ converges to $x^*$. Further, if the corresponding desingularization function is given by $\chi(t) = ct^{\theta}$ (for some $c>0$ and $\theta \in (0,1)$), then the following statements hold:
\begin{enumerate}[label=(\alph*)]
	\item If $\theta \in [1/2, 1)$, then $\{\Rvalue_k\}$ converges R-linearly to $\varphi(x^*)$ and $\{\xk\}$ converges R-linearly to $x^*$. \label{itm1}
	\item If $\theta \in (0, 1/2)$, then there exist constants $\eta_1, \eta_2 > 0$ such that for all $k$ large enough it holds that\label{itm2}
	\begin{align}
		\Rvalue_k - \varphi(x^*) &\leq \eta_1 k^{-\frac{1}{1-2\theta}},\\
		\|\xk - x^*\| &\leq \eta_2 k^{-\frac{\theta}{1-2\theta}}.
	\end{align}
\end{enumerate}
\end{theorem}

\begin{proof}
Taking Theorem~\ref{thm:convergence} into account, we only need to verify statements \ref{itm1} and \ref{itm2}. As a first step, let us prove the results for the sequence $\{\Rvalue_k\}$. We first claim that for $\theta \in (0,1)$ and with
\begin{equation*}
	\omega := \left(\frac{e \taug_C}{2 c \theta b}\right)^2,
\end{equation*}
it holds that
\begin{equation}\label{eq:sigma-inequality}
	\varphi(\xk) - \varphi(x^*) \leq \Big(\frac{1}{\omega}\Big)^{\frac{1}{2(1-\theta)}} \big( \Rvalue_k - \Rvalue_{k+1} \big)^{\frac{1}{2(1-\theta)}}
\end{equation}
for all $k \in \mathbb{N}$ sufficiently large. In fact, if $\varphi(\xk) \leq \varphi(x^*)$ holds, then the left-hand side of \eqref{eq:sigma-inequality}
is nonpositive, hence, the claim holds trivially.
By previous results, we may assume that $k$ is large enough such 
that $\xk \in \clball_\vartheta (x^*)$ and $\varphi(x^*) < \varphi(\xk) < \varphi(x^*) + \eta$ hold.

As $\Phi$ satisfies the KL property at $x^*$ with $\chi(t) = ct^{\theta}$ and as $\Phi(x) = \varphi(x)$ for all $x \in \mathcal{M}$, we have
\begin{align*}
	1 & \leq \chi' \big( \varphi(\xk) - \varphi(x^*) \big)
	\dist \big( 0, \partial \Phi(\xk) \big)\\
	& = c\theta \big( \varphi(\xk) - \varphi(x^*) \big)^{\theta -1}
    \dist \big( 0, \partial \Phi (\xk) \big).
\end{align*}
By \eqref{distbound}, this yields
\begin{equation*}
	1 \leq c\theta \frac{2 b}{\taug_C}  \big( \varphi(\xk) - \varphi(x^*) \big)^{\theta -1} \| \xkp - x^k \|,
\end{equation*}
which gives the inequality
\begin{equation}\label{eq:proofconvrate1}
	\|\xkp - x^k \| \geq \frac{1}{c \theta \frac{2 b}{\taug_C}}
	\big( \varphi(\xk) - \varphi(x^*) \big)^{1-\theta}.
\end{equation}
Further, from \eqref{eq:normxi}, we have
\begin{equation}\label{eq:proofconvrate2}
	\Rvalue_{k+1} - \Rvalue_k \leq - e^2 \|\xkp - \xk \|^2 .
\end{equation}
Combination of \eqref{eq:proofconvrate1} and 
\eqref{eq:proofconvrate2} yields
\begin{align*}
	\Rvalue_{k+1} - \Rvalue_k &\leq - e^2 \|\xkp - x^k \|^2\\
	 & \leq - e^2 \frac{1}{c^2 \theta^2 \frac{4 b^2}{\taug_C^2}} 
	\big( \varphi(\xk) - \varphi(x^*) \big)^{2(1-\theta)}\\
	& = -\omega \big( \varphi(\xk) - \varphi(x^*) \big)^{2(1-\theta)} .
\end{align*}
Rearranging these terms shows that the claim \eqref{eq:sigma-inequality}
holds.

Next recall that, by the acceptance criterion for the stepsize
$ \tauk $, we always have $ \varphi(\xkp) \leq \Rvalue_k $.
Hence, it follows that
\begin{equation}\label{eq:Rvalue-inequality}
	\Rvalue_{k+1} = (1-p_{k+1}) \Rvalue_k + p_{k+1} \varphi(\xkp) \leq (1-\pmin) \Rvalue_k + \pmin \varphi(\xkp).
\end{equation}
Denote by $\{s_k\}$ the sequence defined by $s_k := \Rvalue_k - \varphi(x^*)
\geq 0 $. Then $ s_k \to 0 $ monotonically, and we obtain
\begin{align*}
	s_{k+1} &\leq (1-\pmin) s_k + \pmin(\varphi(\xkp) - \varphi(x^*))\\
			&\leq (1-\pmin)s_k + \pmin \Big(\frac{1}{\omega}\Big)^{\frac{1}{2(1-\theta)}} \big( s_{k+1} - s_{k+2} \big)^{\frac{1}{2(1-\theta)}} ,
\end{align*}
where the first inequality follows from \eqref{eq:Rvalue-inequality}
and the second one results from \eqref{eq:sigma-inequality}.
	By the monotonicity of $\{s_k\}$, this implies
\begin{equation*}
	s_{k+2} \leq (1- \pmin) s_k + \pmin \Big(\frac{1}{\omega}\Big)^{\frac{1}{2(1-\theta)}} \big( s_{k} - s_{k+2} \big)^{\frac{1}{2(1-\theta)}},
\end{equation*}
which, in turn, yields
\begin{align*}
s_k &\leq \frac{1}{\pmin} (s_k -s_{k+2}) + \Big(\frac{1}{\omega}\Big)^{\frac{1}{2(1-\theta)}} \big( s_k - s_{k+2} \big)^{\frac{1}{2(1-\theta)}}\\
	&\leq \left(\frac{1}{\pmin} + \left(\frac{1}{\omega}\right)^{\frac{1}{2(1-\theta)}}\right) (s_k - s_{k+2})^{\min\{1, \frac{1}{2(1-\theta)}\}}
\end{align*}
for all $ k $ sufficiently large.
As for all $a, b > 0$ it holds that $1/\min\{a, b\} = \max\{1/a, 1/b\}$, it follows that
\begin{equation*}
s_k^{\max\{1, 2(1-\theta)\}} \leq \gamma (s_k - s_{k+2}),
\end{equation*}
where 
\begin{equation*}
\gamma := \left(\frac{1}{\pmin} + \left(\frac{1}{\omega}\right)^{\frac{1}{2(1-\theta)}}\right)^{\max\{1, 2(1-\theta)\}} > 0
\end{equation*}
is a constant.

By considering even and odd indices separately, we are now in the setting of Lemma~\ref{Lem:Aragon} and immediately obtain the 
corresponding rate-of-convergence results for the sequence $\{\Rvalue_k\}$ as $\theta \in (0, 1/2)$ implies that $2(1-\theta) > 1$ and $\theta \in [1/2, 1)$ implies that $2(1-\theta) \in (0, 1]$.

Let us now verify the statements for the sequence $\{\xk\}$. In view of Theorem~\ref{thm:convergence}, the equations in the proof of that result remain valid if $k_0-1$ is replaced by some $k$ sufficiently large. Note that $\Xi_{j-1} \leq \sqrt{\Rvalue_{j-1} - \varphi(x^*)} = \sqrt{s_{j-1}}$. Taking the monotonicity of the function $\chi$ and the sequences $\{\Rvalue_k\}$ and thus $\{s_k\}$ into account, it follows from \eqref{eq:sumXis} in combination with \eqref{eq:normxi} and the definition of $\chi$ that, for $l>k$:
\begin{align*}
	\|\xk - x^l\| &\leq \sum_{j=k}^{l-1} \|x^{j+1}-x^j\| \leq \frac{1}{e} \sum_{j=k}^{l-1} \Xi_j  \\
	&\leq \frac{1}{e} \sum_{j=k+1}^{l(k+1)} (3\Xi_{j-1} + \Xi_j) + \frac{2 \hat c}{e} \sum_{j = k+1}^{l(k+1)}\chi \big( \Rvalue_j - \varphi(x^*) \big)\\
	&\leq \frac{4}{e} \sum_{j=k+1}^{l(k+1)} \sqrt{s_{j-1}} + \frac{2 \hat c}{e} \sum_{j = k+1}^{l(k+1)}\chi ( s_j )\\
	&\leq \frac{4m}{e} \sqrt{s_k} + \frac{2 \hat c m}{e}\chi ( s_k )\\
	&\leq \tilde \eta s_k^{\min\{1/2, \theta \}}
\end{align*}
for all $ k $ sufficiently large, where 
\begin{equation*}
\tilde \eta := \frac{4 + 2 \hat c c}{e}m.
\end{equation*}
Taking now $l \to \infty$, together with the corresponding rate-of-convergence results for $\{s_k\}$ from the first part, this completes the proof.
\end{proof}

\end{document}